\documentclass{amsart}
\usepackage{latexsym,amsxtra,amscd,ifthen}
\usepackage{amsfonts}
\usepackage{verbatim}
\usepackage{amsmath}
\usepackage{amsthm}
\usepackage{amssymb}
\usepackage{url}

\theoremstyle{plain}

\newtheorem{theorem}{Theorem}
\newtheorem{lemma}[theorem]{Lemma}
\newtheorem{proposition}[theorem]{Proposition}
\newtheorem{corollary}[theorem]{Corollary}

\numberwithin{theorem}{section}
\numberwithin{equation}{theorem}
\theoremstyle{definition}

\newtheorem{definition}[theorem]{Definition}
\newtheorem{example}[theorem]{Example}
\newtheorem{remark}[theorem]{Remark}
\newtheorem{question}[theorem]{Question}
\newtheorem*{question*}{Question}

\newcommand{\cwlt}{(\textup{cwlt})}

\newcommand{\INT}{\textup{int}}

\DeclareMathOperator{\rk}{rk}

\DeclareMathOperator{\End}{End}

\DeclareMathOperator{\GKdim}{GKdim}

\DeclareMathOperator{\Aut}{Aut}
\DeclareMathOperator{\gr}{gr}
\DeclareMathOperator{\tr}{tr}

\DeclareMathOperator{\ML}{ML}
\DeclareMathOperator{\LND}{LND}

\begin{document}

\title
{Zariski Cancellation Problem for Noncommutative Algebras}

\author{Jason Bell and James J. Zhang}

\address{Bell: Department of  Pure Mathematics,
University of Waterloo, Waterloo, ON N2L 3G1,
Canada}

\email{jpbell@uwaterloo.ca}

\address{Zhang: Department of Mathematics, Box 354350,
University of Washington, Seattle, Washington 98195, USA}
\email{zhang@math.washington.edu}

\begin{abstract}
A noncommutative analogue of the Zariski cancellation 
problem asks whether $A[x]\cong B[x]$ implies $A\cong B$ 
when $A$ and $B$ are noncommutative algebras.  We resolve this 
affirmatively in the case when $A$ is a noncommutative finitely 
generated domain over the complex field of Gelfand-Kirillov 
dimension two.  In addition, we resolve the Zariski cancellation 
problem for several classes of Artin-Schelter regular algebras of 
higher Gelfand-Kirillov dimension.
\end{abstract}

\subjclass[2000]{Primary 16W70, 16W25, 16S38}


\keywords{Zariski cancellation problem, noncommutative algebra,
effective and dominating discriminant, locally nilpotent derivation,
skew polynomial ring}


\maketitle

\setcounter{section}{-1}
\section{Introduction}
\label{xxsec0}

Kraft said in his 1995 survey \cite{Kr}
that ``{\it there is no doubt that  complex affine
$n$-space ${\mathbb A}^n={\mathbb A}_{\mathbb C}^n$ is one of the basic
objects in algebraic geometry. It is therefore surprising how little
is known about its geometry and its symmetries. Although there has 
been some remarkable progress in the last few years, many basic problems
remain open.}'' His remark still applies even today---20 years later.
Let us start with one of the famous questions in commutative affine 
geometry. Throughout the introduction, we let $k$ be an algebraically 
closed field of characteristic zero (except for some results mentioned 
below).

\begin{question}[{\bf Zariski Cancellation Problem}]
\label{xxque0.1}
Does an isomorphism $Y\times {\mathbb A}^{1}\cong {\mathbb A}^{n+1}$
imply that $Y$ is isomorphic to ${\mathbb A}^{n}$? Or equivalently,
does an isomorphism $B[t]\cong k[t_1,\cdots,t_{n+1}]$ of algebras 
imply that $B$ is isomorphic to $k[t_1,\cdots,t_n]$?
\end{question}

For simplicity, let {\bf ZCP} denote the Zariski Cancellation Problem.
An algebra $A$ is called {\it cancellative} if $A[t]\cong B[t]$ for 
some algebra $B$ implies that $A\cong B$. So the {\bf ZCP} asks if the 
commutative polynomial ring $k[x_1,\cdots,x_n]$ is cancellative. Recall 
that $k[x_1]$ is cancellative by a result of Abhyankar-Eakin-Heinzer 
\cite{AEH}, while $k[x_1,x_2]$ is cancellative by Fujita \cite{Fu} and 
Miyanishi-Sugie \cite{MS} in characteristic zero and by Russell 
\cite{Ru} in positive characteristic. The {\bf ZCP} was open for 
many years. In 2013, a remarkable development was made by 
Gupta \cite{Gu1,Gu2} who completely settled the {\bf ZCP} negatively 
in positive characteristic for $n\geq 3$. The {\bf ZCP} in 
characteristic zero remains open for $n\geq 3$. We give a list of 
open questions and problems that are closely related to the {\bf ZCP}.

\begin{question}
\label{xxque0.2}
For the following, let $k^\times$ be $k\setminus \{0\}$.
\begin{enumerate}
\item[]
{\rm{\bf (ChP:=Characterization Problem)}}
Find an algebro-geometric characterization of ${\mathbb A}^n$.
\item[]
{\rm{\bf (EP:=Embedding Problem)}}
Is every closed embedding ${\mathbb A}^{a}\hookrightarrow
{\mathbb A}^{a+n}$ equivalent to the standard embedding?
\item[]
{\rm{\bf (AP:=Automorphism Problem)}}
Describe the group of polynomial 
automorphisms of ${\mathbb A}^n$.
\item[]
{\rm{\bf (LP:=Linearization Problem)}}
Is every automorphism of ${\mathbb A}^n$ of finite order linearizable?
\item[]
{\rm{\bf (JC:=Jacobian Conjecture)}}
Is every polynomial morphism $\phi: {\mathbb A}^n
\to {\mathbb A}^n$
with $\det J\phi\in k^\times$ an isomorphism?
\end{enumerate}
\end{question}

There are some known relationships between these problems. For example, 
a positive solution of the {\bf LP} would imply a positive solution 
of the {\bf ZCP}. When $n\leq 2$, most of these questions (except for 
the {\bf JC}) are resolved and there is a diagram of implications
$$
{\bf EP} \Longrightarrow {\bf AP}  \Longrightarrow {\bf LP}  
\Longrightarrow {\bf ZCP}
$$
along with a  possible ``missing link" ${\bf JC} \Longrightarrow {\bf ZCP}$ 
(see \cite{vdE}). Note that the {\bf EP} (in dimension 2) was solved by 
Abyhyankar-Moh \cite{AM} and Suzuki \cite{Su}. Gupta's work 
\cite{Gu1, Gu2} would suggest a negative solution to the {\bf ZCP}, 
even in characteristic zero. If the ``missing link" could be established
and if the {\bf ZCP} had a negative solution, then the {\bf JC} could 
be settled negatively. Many authors have been working on these 
questions---see the references in \cite{Kr, Lu, vdE}.

Some naive and direct translations of these questions into the 
noncommutative setting are easily seen to have negative solutions. So it is 
important to carefully formulate noncommutative versions of these 
questions and to understand for which classes of (commutative or 
noncommutative) algebras these questions have positive or 
negative answers. Hopefully new ideas will emerge via the study of 
the noncommutative versions of these open questions. In this paper 
we mainly consider the following noncommutative formulation of the 
{\bf ZCP}.

\begin{question}
\label{xxque0.3}
Let $A$ be a noncommutative noetherian Artin-Schelter regular algebra
\cite{AS}. When is $A$ cancellative?
\end{question}

Since Artin-Schelter regular algebras are considered as a noncommutative
generalization of the commutative polynomial rings, the above question 
can be viewed as a noncommutative version of {\bf ZCP}.

In this paper we present two ideas to deal with the 
{\bf ZCP} for some families of noncommutative algebras. One
is to use the \emph{Makar-Limanov invariant} and the other is to use 
 \emph{discriminants}. 

Let us first review the Makar-Limanov invariant.
Let $A$ be an algebra and let $\LND(A)$ 
be the collection of locally nilpotent $k$-derivations of $A$. The 
{\it Makar-Limanov invariant} of $A$ is defined to be 
$$\ML(A)=\bigcap_{\delta\in \LND(A)} \ker (\delta).$$
The Makar-Limanov invariant was originally introduced by 
Makar-Limanov \cite{Ma1} and has become a very useful invariant in 
commutative algebra. We say that $A$ is {\it $\LND$-rigid} if 
$\ML(A)=A$, or equivalently if $\LND(A)=\{0\}$. One of our main results 
(see Theorem \ref{xxthm3.6} for the precise statement and proof) is the 
following, which shows that rigidity controls cancellation.

\begin{theorem}
\label{xxthm0.4} 
Let $A$ be a finitely generated domain of finite Gelfand-Kirillov 
dimension. If $A$ is $\LND$-rigid, then $A$ is cancellative. 
\end{theorem}

By the above theorem, we would like to show that various classes 
of noncommutative algebras are $\LND$-rigid. Here is one 
of the consequences [Corollary \ref{xxcor3.7}].

\begin{theorem}
\label{xxthm0.5}
Let $k$ be an algebraically closed field of characteristic zero. 
Let $A$ be a finitely generated domain of Gelfand-Kirillov dimension two.
If $A$ is not commutative, then $A$ is cancellative.
\end{theorem}

By \cite{AEH}, every commutative domain of Gelfand-Kirillov dimension
(or GK-dimension, for short) one is cancellative. By \cite{Da, Fi} 
there are commutative domains of GK-dimension two that are not cancellative. 
Theorem \ref{xxthm0.5} ensures that every {\bf non}-commutative
domain of GK-dimension two is cancellative. Crachiola \cite{Cr} showed 
that commutative UFDs of GK-dimension two are always cancellative. 

Next let us talk about the discriminant method. The discriminant method 
was introduced in \cite{CPWZ1, CPWZ2} to answer the {\bf AP} for a class 
of noncommutative algebras. The definition of  the discriminant 
in the noncommutative setting will be reviewed in Section \ref{xxsec4}. 
Suppose that $A$ is finitely generated by $Y=\oplus_{i=1}^d kx_i$ as 
an algebra. An element $f\in A$ is 
called {\it effective}, if for every testing ${\mathbb N}$-filtered
$k$-algebra $T$ with $\gr T:=\bigoplus F_i T/ F_{i-1} T$ 
being an ${\mathbb N}$-graded domain, and for every testing subset 
$\{y_1,\ldots,y_d\} \subset T$ satisfying (a) it is linearly 
independent in the quotient $k$-module $T/k 1_T$ and (b) some $y_i$ is 
not in $F_0 T$, there is a presentation of $f$ of the form 
$f(x_1,\ldots,x_d)$ when lifted in the free algebra $k\langle x_1,
\ldots,x_d\rangle$ such that $f(y_1,\cdots,y_d)$ is either zero 
or not in $F_0 T$. For example, each monomial $x_1^{a_1}\cdots x_d^{a_d}$, 
for some positive integers $a_1,\ldots,a_d$, is effective. Note that 
there are non-monomial effective discriminants (see Examples 
\ref{xxex5.5} and \ref{xxex5.6}). Here is one of our main results
by using the discriminant, which provides a uniform way of showing 
the rigidity for some classes of noncommutative algebras.

\begin{theorem} 
\label{xxthm0.6}
Suppose that $A$ is a domain which is a finitely generated module over 
its affine center $C$ and that the discriminant $d(A/C)$ is effective. 
Then $A$ is cancellative.
\end{theorem}

The above theorem does not answer the original {\bf ZCP} as, when $A$ 
is commutative, the discriminant over its center is trivial and not 
effective. However, Theorem \ref{xxthm0.6} applies to a large family of 
noncommutative algebras. One can check, for example, that the skew 
polynomial ring $k_q[x_1,\ldots, x_{n}]$ where $n$ is 
even and $1\neq q$ is a root of unity has effective discriminant.
Then, by Theorem \ref{xxthm0.6}, $k_q[x_1,\cdots,x_n]$ is cancellative. 
The next result shows a connection between the noncommutative 
{\bf ZCP} and the noncommutative {\bf AP}. Let $C$ 
denote the center of the algebra $A$ and we refer to Definition 
\ref{xxdef4.5} for the definition of  ``dominating''.   We have 
the following result (see Theorem \ref{xxthm5.7} for an expanded version).

\begin{theorem}
\label{xxthm0.7} 
Let $A$ be a skew polynomial ring $k_{p_{ij}}[x_1,\cdots,x_n]$
where each $p_{ij}$ is a root of unity. The following are equivalent.
\begin{enumerate}
\item[(1)]
The full automorphism group $\Aut(A)$ is affine \cite[Definition 2.5]{CPWZ1}.
\item[(2)]
The discriminant $d(A/C)$ is dominating.
\item[(3)]
The discriminant $d(A/C)$ is effective.
\item[(4)]
$A$ is $\LND$-rigid.
\end{enumerate}
Consequently, under any of these equivalent conditions, 
$A$ is cancellative.
\end{theorem}

In general, by using the Makar-Limanov invariant and Theorem \ref{xxthm0.4}, 
we show that if $d(A/C)$ is dominating, then $A$ is 
cancellative, see Theorem \ref{xxthm4.7}(2). As an example, we 
have the following result.

\begin{theorem}
\label{xxthm0.8} Let $A$ be a finite tensor product of algebras
of the form
\begin{enumerate}
\item
$k_{p}[x_1,\cdots,x_n]$ where $1\neq p\in k^\times$ and $n$ is even;
\item
$k\langle x, y\rangle /(x^2y-yx^2, y^2x+xy^2)$;
\item
$k\langle x, y\rangle/(yx-q xy-1)$ where $1\neq q\in k^\times$.
\end{enumerate}
Then $A$ is $\LND$-rigid. As a consequence, $A$ is cancellative.
\end{theorem}

\begin{remark}
\label{xxrem0.9} Suppose that $n$ is odd and that $q\neq 1$ is a root 
of unity. It is an open question whether $k_q[x_1,\cdots,x_n]$ is 
cancellative. There are two results related to this. 
\begin{enumerate}
\item[(1)]
The following weak cancellative property holds as a
consequence of \cite[Theorem 9]{BZ}:
Let $B$ be a connected graded algebra generated in degree one. If $k_q[x_1,\cdots,x_n] [t]
\cong B[t]$ as algebras, then $k_q[x_1,\cdots,x_n]\cong B$ as graded algebras. 
\item[(2)]
A result of \cite{CYZ2} says that Veronese subrings 
of $k_q[x_1,\cdots,x_n]^{(v)}$ is cancellative when $m$ and $v$ are not
coprime, where $m$ is the order of $q$. 
\end{enumerate}
\end{remark}

\subsection*{Acknowledgments} 
The authors would like to thank the referee for his/her very careful 
reading and extremely valuable comments.
J. Bell was supported by NSERC grant NSERC RGPIN-326532 and 
J.J. Zhang by the US National Science 
Foundation (NSF grant Nos. DMS-0855743 and DMS-1402863).

\section{Trivial center vs. cancellation}
\label{xxsec1}

Throughout the rest of the paper we let $k$ be a base commutative domain. 
Sometimes we further assume that $k$ is a field. 
Everything is taken over $k$, for example, $\otimes$ stands for 
$\otimes_k$. We sometimes consider $k$-flat algebras. 
If $k$ is a field, then every $k$-module is flat. 
First we recall the definition of cancellative.

\begin{definition}
\label{xxdef1.1}
Let $A$ be an algebra.
\begin{enumerate}
\item
We call $A$ {\it cancellative} if $A[t]\cong B[t]$ for
some algebra $B$ implies that $A\cong B$.
\item
We call $A$ {\it strongly cancellative} if, for each $d\geq 1$, 
$A[t_1,\ldots,t_d]\cong B[t_1,\ldots,t_d]$ for some algebra 
$B$ implies that $A\cong B$.
\item
We call $A$ {\it universally cancellative} if, for every 
$k$-flat finitely generated commutative domain $R$ 
such that the natural map $k\to R\to R/I$ is an isomorphism 
for some ideal $I\subset R$ and every  
$k$-algebra $B$, $A\otimes R\cong B\otimes R$ implies that 
$A\cong B$.
\end{enumerate}
\end{definition}

\begin{remark}
\label{xxrem1.2}
By the above definition, it is easy to see that
\begin{center}
universally cancellative $\quad \Longrightarrow \quad $ strongly cancellative
$\quad \Longrightarrow\quad $ cancellative.
\end{center}
But, it is not obvious to us whether any two of them are equivalent.
\end{remark}

We have an immediate observation for the noncommutative 
cancellation problem. Let $C(A)$ denote the center of
$A$. 

If $A$ is an algebra over a commutative base ring $k$ 
(which we assume to be a domain but not a 
field in general), then the Gelfand-Kirillov dimension 
(or GK-dimension, for short) of $A$ is defined to be
\begin{equation}
\label{E1.2.1}\tag{E1.2.1}
\GKdim A=\sup_V \left(\overline{\lim_{n\to \infty}} 
\log_{n} {\text{rank}}_k(V^n)\right)
\end{equation}
where $V$ varies over all finitely generated $k$-submodules 
of $A$, and the rank of a finitely generated $k$-module $M$ 
is defined to be the dimension of $M\otimes_k {\rm Frac}(k)$ 
as a ${\rm Frac}(k)$-vector space, where ${\rm Frac}(k)$ is 
the field of fractions of $k$. We refer to the book \cite{KL} 
for basic properties of Gelfand-Kirillov dimension.

\begin{proposition}
\label{xxpro1.3}
Let $k$ be a field and $A$ be an algebra with center being $k$. 
Then $A$ is universally cancellative.
\end{proposition}

\begin{proof} For any algebra $A$, let $C(A)$ denote the 
center of $A$. Let $R$ be an affine commutative domain such 
that $R/I=k$ for some ideal $I\subset R$ and suppose
that $\phi: A\otimes R\to B\otimes R$ is an algebra 
isomorphism for some algebra $B$. 
Since $C(A)=k$, we have $C(A\otimes R)=R$. 
Since $C(B\otimes R)=C(B)\otimes R$ and since $\phi$ induces
an isomorphism between the centers, we have 
\begin{equation}
\label{E1.3.1}\tag{E1.3.1}
R\cong C(B)\otimes R.
\end{equation} 
Consequently, $C(B)$ is a commutative domain. 
By considering the GK-dimension of both sides of \eqref{E1.3.1}, 
one sees that $\GKdim C(B)=0$, when regarded as a $k$-algebra.
(This also follows from Lemma \ref{xxlem3.1}(2).)
Hence $C(B)$ is a field. Since there is an ideal
$I$ such that $R/I=k$, $C(B)=k$. Consequently,
$C(B\otimes R)=R$. Now the induced map $\phi$ is an
isomorphism between $C(A\otimes R)=R$ to 
$C(B\otimes R)=R$, so we have $R/\phi(I)=k$. 
Finally,  $\phi$ induces an automorphism from 
$A\cong A\otimes R/(I)\cong B\otimes R/(\phi(I))\cong B.$
\end{proof}

We list some easy consequences below.

\begin{example}
\label{xxex1.4} 
We have the following results.
\begin{enumerate}
\item[(1)]
Let $k$ be a field of characteristic zero and 
$A_n$ the $n$th Weyl algebra. Then $C(A_n)=k$.
So $A_n$ is universally cancellative.
\item[(2)]
Let $k$ be a field and $q\in k^\times$.
Let $k_q[x_1,\ldots,x_n]$ be the skew 
polynomial ring generated by $x_1,\ldots,x_n$ subject
to the relations $x_jx_i=q x_ix_j$ for all $1\leq i<j\leq n$. 
If $n\geq 2$ and $q$ is not a root of unity, then 
$C(A)=k$. So $A$ is universally cancellative.
\end{enumerate}
\end{example}

\section{Higher derivations and Makar-Limanov invariant}
\label{xxsec2}

The Makar-Limanov invariant is a very useful invariant 
to deal with the cancellation problem.  We will also use a 
modified version of Makar-Limanov invariant to better
control the cancellation in positive characteristic.
Given a $k$-algebra $A$, let ${\rm Der}(A)$ denote the 
collection of $k$-derivations of $A$ and let $\LND(A)$ denote the 
collection of locally nilpotent $k$-derivations of $A$. 

For a sequence of $k$-linear endomorphisms 
$\partial:=\{\partial_i\}_{i\geq 0}$ of $A$ with the 
property that for every $a\in A$ we have $\partial_i(a)=0$ 
for $i$ sufficiently large,
and for every $c\in k$, we define
\begin{equation}
\label{E2.0.1}\tag{E2.0.1}
G_{c\partial}: A\to A \quad {\rm{by}}\quad 
a\to \sum_{i=0}^{\infty} c^i \partial_i(a) \quad 
\end{equation}
and 
\begin{equation}
\label{E2.0.2}\tag{E2.0.2}
G_{\partial,t}: A[t]\to A[t] \quad {\rm{by}}\quad 
a\to \sum_{i=0}^{\infty} \partial_i(a)t^i, t\to t,
\end{equation}
for all $a\in A$.

\begin{definition}
\label{xxdef2.1} Let $A$ be an algebra.
\begin{enumerate}
\item[(1)]\cite[Definition 2.3(1)]{LL}
A {\it higher derivation} (or {\it Hasse-Schmidt derivation}
\cite{HS})  on $A$ is a sequence of $k$-linear endomorphisms
$\partial:=\{\partial_i\}_{i=0}^{\infty}$ such that:
$$\partial_0 = id_A, \quad 
\text{and} \quad  \partial_n(ab) =\sum_{i=0}^n 
\partial_i(a)\partial_{n-i}(b)
$$
for all $a, b \in A$ and all $n\geq 0$. The collection of
higher derivations is denoted by ${\rm Der}^H(A)$. 
\item[(2)]\cite[Definition 2.3(1)]{LL}
A higher derivation is called {\it iterative} if 
$\partial_i \partial_j ={i+j \choose i} \partial_{i+j}$ 
for all $i, j \geq 0$. 
\item[(3)]
A higher derivation is called {\it locally nilpotent} if 
\begin{enumerate}
\item
for all $a \in A$  there exists $n \geq  0$ such that $\partial_i(a) = 0$ 
for all $i \geq  n$,
\item
the map $G_{\partial,t}$ defined in \eqref{E2.0.2} is an algebra automorphism
of $A[t]$.
\end{enumerate}
The collection of locally nilpotent higher derivations is denoted by 
$\LND^H(A)$ and the collection of locally nilpotent iterative higher 
derivations is denoted by $\LND^I(A)$.
\item[(4)]
For every $\partial\in {\rm Der}^H(A)$, the kernel of $\partial$
is defined to be 
$$\ker \partial =\bigcap_{i\geq 1} \ker \partial_i.$$ 
\end{enumerate}
\end{definition}

Higher derivations  have been well-studied and had many applications
\cite{HS, LL, Mat, NM}. 
We repeat some elementary comments given in \cite[Remark 2.4]{LL}.
Given a higher derivation $\partial=(\partial_i)_{i\geq 0}$, 
$\partial_1$ is necessarily a derivation of $A$. Hence 
there is a map ${\rm Der}^H(A)\to {\rm Der}(A)$ by sending
$\partial$ to $\partial_1$. In characteristic $0$, the only 
iterative higher derivation $\partial=(\partial_i)$ 
on $A$ such that $\partial_1=\delta$ is given by:
\begin{equation}
\label{E2.1.1}\tag{E2.1.1}
\partial_n=\frac{\delta^n}{n!}
\end{equation}
for all $n \geq  0$. This iterative higher derivation is called 
the canonical higher derivation associated to $\delta$.
In this case, we have a map ${\rm Der}(A)\to {\rm Der}^H(A)$
sending $\delta$ to $(\partial_i)$ as defined by \eqref{E2.1.1}. Hence
the map ${\rm Der}(A)\to {\rm Der}^H(A)$ is the right inverse of 
the map ${\rm Der}^H(A)\to {\rm Der}(A)$. In positive characteristic,
by \cite[Remark 2.4(2)]{LL}, an iterative
higher derivation is not uniquely determined by $\partial_1$. 

It is not clear to us if Definition \ref{xxdef2.1}(3b) is a consequence of 
Definition \ref{xxdef2.1}(3a). See \cite[(1.6)]{Mat} for the inverse of
$G_{\partial, t}$ if it is invertible. But $G_{\partial, t}$ is invertible 
when $\partial$ is iterative, as proven in part (2) of the following lemma.

\begin{lemma}
\label{xxlem2.2} Let $\partial:=(\partial_i)_{i\geq 0}$ be a 
 higher derivation.
\begin{enumerate}
\item[(1)]
Suppose $\partial$ is locally nilpotent. 
For every $c\in k$, $G_{c\partial}$ is an algebra automorphism of $A$.
\item[(2)]
If $\partial$ is iterative and satisfies Definition {\rm{\ref{xxdef2.1}(3a)}}, 
then $G_{\partial, t}$ is an algebra  automorphism of $A[t]$.
As a consequence, $\partial$ is locally nilpotent.
\item[(3)]
If $G: A[t]\to A[t]$ be a $k[t]$-algebra automorphism and 
if $G(a)\equiv a \mod (t)$ for all $a\in A$, then $G=G_{\partial, t}$
for some $\partial\in \LND^H(A)$.  
\end{enumerate}
\end{lemma}

\begin{proof}
(1) By definition, $G_{\partial, t}$ is an algebra automorphism
of $A[t]$ and $G_{\partial, t}(t-c)=t-c$. Hence, it induces an algebra
automorphism of $A[t]/(t-c)$. The induced automorphism is indeed $G_{c\partial}$.

(2) First we show that $G_{\partial,t}$ is a $k[t]$-algebra homomorphism, or
equivalently, $G_{\partial,t}(ab)=G_{\partial,t}(a)G_{\partial,t}(b)$ for 
all $a,b\in A[t]$. In fact, it suffices to show this equation for 
all $a,b\in A$, by $k[t]$-linearity. Observe that, for all $a,b\in A$,
\begin{eqnarray*} 
G_{\partial,t}(ab) &=& \sum_{i=0}^{\infty} t^i \partial_i(ab) \\
&=& \sum_{i=0}^{\infty} t^i \left(\sum_{j=0}^i \partial_j(a)\partial_{i-j}(b)\right) \\
&=& \sum_{j=0}^{\infty} \partial_j(a)t^j \left( \sum_{i=j}^{\infty} t^{i-j} 
\partial_{i-j}(b) \right)\\
&=& G_{\partial,t}(a)G_{\partial,t}(b),\end{eqnarray*}
where all interchanging of summations can be justified by the fact that the sums 
are actually finite.  To see that $G_{\partial,t}$ is an automorphism, note that 
$k[t]$-linearity of $G_{\partial,t}$ and iterativity of $\partial$ give
\begin{eqnarray*}
G_{\partial,t}\circ G_{\partial,-t}(a) 
&=& G_{\partial,t} \left( \sum_{i=0}^{\infty} (-t)^i \partial_i(a) \right) \\
&=&  \sum_{i=0}^{\infty} (-t)^i G_{\partial,t}(\partial_i(a)) \\
&=&  \sum_{i=0}^{\infty} (-t)^i \left( \sum_{j=0}^{\infty} t^j \partial_j \partial_i(a) \right) \\
&=& \sum_{i=0}^{\infty} (-t)^i \left( \sum_{j=0}^{\infty} t^j {i+j\choose j} \partial_{i+j}(a)\right) \\
&=& \sum_{i=0}^{\infty} \sum_{n=i}^{\infty} {n\choose i} (-1)^i t^n \partial_n(a) \\
&=& \sum_{n=0}^{\infty} t^n \partial_n(a) \left(\sum_{i=0}^n (-1)^i {n\choose i} \right) \\
&=& \partial_0(a) \\
&=& a,\end{eqnarray*}
which gives that $G_{\partial,t}$ is invertible. 

(3) Write $G(a)=\sum_{i\geq 0} \partial_i(a) t^i$
for all $a\in A$. Similar to the first part of the proof of (2), 
one sees that $\partial:=(\partial_i)$ is in $\LND^H(A)$.
\end{proof}

We now recall the definition of the Makar-Limanov invariant.

\begin{definition}
\label{xxdef2.3}
Let $A$ be an algebra over $k$. Let $*$ be either blank, or $^H$, or $^I$. 
\begin{enumerate}
\item[(1)]
The {\it Makar-Limanov$^*$ invariant} \cite{Ma1} of $A$ is defined to be 
\begin{equation}
\label{E2.3.1}\tag{E2.3.1}
{\rm ML}^*(A) \ = \ \bigcap_{\delta\in {\rm LND}^*(A)} {\rm ker}(\delta).
\end{equation}
This means that we have original $\ML(A)$, as well as, $\ML^H(A)$
and $\ML^I(A)$.
\item[(2)]
We say that $A$ is \emph{$\LND^*$-rigid} if ${\rm ML}^*(A)=A$,
or $\LND^*(A)=\{0\}$.
\item[(3)]
$A$ is called \emph{strongly $\LND^*$-rigid} if ${\rm ML}^*(A[t_1,\ldots,t_d])=A$,
for all $d\geq 1$.
\end{enumerate}
\end{definition}

\begin{example}
\label{xxex2.4}
Let $T$ be the polynomial ring $A[t_1,\cdots,t_d]$ over some $k$-algebra
$A$. We fix an integer $1\leq i\leq d$. For each $n\geq 0$, define a 
divided power version of $A$-linear differential operator
\begin{equation}
\label{E2.4.1}\tag{E2.4.1}
\Delta_i^n: t_1^{m_1}\cdots t_d^{m_d} \longrightarrow 
\begin{cases} {m_i \choose n} t_1^{m_1}\cdots t_i^{m_i-n}
\cdots t_d^{m_d} & {\rm{if}} \;\; m_i\geq n\\
0 & {\rm{otherwise,}}
\end{cases}
\end{equation}
where ${m_i \choose n}$ is defined in ${\mathbb Z}$ or
in ${\mathbb Z}/(p)$. Then $\{\Delta_i^n\}_{n= 0}^{\infty}$
is a locally nilpotent iterative higher derivation of $T$. 

If $k$ contains ${\mathbb Q}$, then $\Delta_i^n$ agrees with 
$\frac{1}{n !} (\frac{\partial}{\partial t_i})^n$ for all $i$ and $n$. 

Using locally nilpotent iterative higher derivations 
$\partial_i:=\{\Delta_i^n\}_{n=0}^{\infty}$ of $T$, for $i=1,\cdots,d$,
one sees that $\ML^H(T)\subseteq \ML^I(T)\subseteq A$. 
In particular, 
$$\ML^H(k[t_1,\cdots,t_d])=\ML^I(k[t_1,\cdots,t_d])=k.$$
\end{example}

\begin{remark}
\label{xxrem2.5} Let $A$ be a $k$-algebra. 
\begin{enumerate}
\item[(1)]
Suppose $k$ contains ${\mathbb Q}$. By using \eqref{E2.1.1}, 
one sees that there is a bijection between $\LND^I(A)$ and
$\LND(A)$. As a consequence, $\ML^I(A)=\ML(A)$. 
Since $\LND^I(A)\subseteq \LND^H(A)$, we obtain that 
$\ML^H(A)\subseteq \ML(A)$. In particular, if $A$ is $\LND^H$-rigid,
then it is $\LND$-rigid. 
\item[(2)]
Suppose $k$ contains ${\mathbb Q}$.
It is not obvious to us whether
$\ML^H(A)=\ML(A)$ in general. In particular, we don't know if
$\LND$-rigidity is equivalent to $\LND^H$-rigidity.
\item[(3)]
Suppose the prime field of $k$ is finite, but not
${\mathbb F}_2$. Let 
$A$ be the skew polynomial ring $k_{-1}[x_1,x_2]$ and 
$\partial$ be the nonzero locally nilpotent derivation of
$A$ given in \cite[Example 3.9]{CPWZ1}. Then, by definition, 
$\ML(A) \subsetneq A$. On the other hand, by Theorem \ref{xxthm4.7}(1)
and Example \ref{xxex4.8}(1) in Section \ref{xxsec4}, 
$A$ is $\LND^H$-rigid, namely, $\ML^H(A)=A$. 
Therefore
$$\ML^I(A)=\ML^H(A)=A\supsetneq  \ML(A).$$ 
In particular, $\LND^H$-rigidity
is not equivalent to $\LND$-rigidity. In this example,
$A$ is (strongly) cancellative, see Theorem \ref{xxthm4.7}(2).
\item[(4)]
It follows from part (c) that the locally nilpotent derivation 
$\partial$ given in \cite[Example 3.9]{CPWZ1} 
(when ${\rm{char}}\; k=p>2$) can not be extended to a 
{\bf locally nilpotent} higher derivation, but it is standard that 
$\partial$ can be extended to an iterative higher derivation
by using an idea similar to \eqref{E2.4.1}. 
\end{enumerate}
\end{remark}

\begin{remark}
\label{xxrem2.6}
Suppose $A$ contains ${\mathbb Z}$.
Let $*$ be either  blank,  $^H$ or $^I$.
\begin{enumerate}
\item[(1)]
It is clear that $\ML^*(A[t_1,\ldots,t_d]) \subseteq \ML^*(A)$, 
but, it is not obvious to us whether 
$\ML^*(A[t_1,\ldots,t_d]) =\ML^*(A)$.
\item[(2)]
Makar-Limanov made the following conjecture in \cite{Ma2}:
If $A$ is a commutative domain over a field of 
characteristic zero, then $\ML(A[t_1,\ldots,t_d])=\ML(A)$.
And he proved that the conjecture holds when $\GKdim A=1$
\cite{Ma2}.
\end{enumerate}
\end{remark}

\section{Rigidity controls cancellation}
\label{xxsec3}
 
We shall investigate the relationship between 
$\LND$-rigidity (respectively, strong $\LND$-rigidity) and cancellation
(respectively, strong cancellation).

We need the following lemma which is \cite[Proposition 3.11 and Lemma 6.5]{KL}
when $k$ is a field. See the definition of Gelfand-Kirillov 
dimension, denoted by $\GKdim$, before Proposition \ref{xxpro1.3}. 
We always assume that the base ring $k$ is a commutative domain. 

\begin{lemma}
\label{xxlem3.1} 
Let $A$ be a $k$-algebra and $R$ be an affine commutative $k$-algebra.
\begin{enumerate}
\item[(1)]
$\GKdim A=\GKdim_Q (A\otimes Q)$ where $Q$ is the field of fractions of
$k$. In particular, if $A$ is affine and commutative, $\GKdim A$ is an 
integer.
\item[(2)]\cite[Proposition 3.11]{KL}
$\GKdim A\otimes R=\GKdim A+\GKdim R$.
\item[(3)]\cite[Lemma 6.5]{KL}
Let $\{F_i A\}_{i\in {\mathbb Z}}$ be a filtration of $A$ in the sense of 
\cite[p.73]{KL}. Let $M$ be a filtered right $A$-module with 
filtration $\{F_i M\}_{i\in {\mathbb Z}}$ in the sense of 
\cite[p.74]{KL}. Then $\GKdim \gr (M)\leq \GKdim M$.
\end{enumerate}
\end{lemma}

\begin{proof} (1) This follows from the definition of $\GKdim$ 
\eqref{E1.2.1} and the equation
$${\rm{rank}}_{k} (V^n)=\dim_{Q} (V\otimes_k Q)^n.$$

(2) This follows from part (1) and \cite[Proposition 3.11]{KL}.

(3) This follows from part (1) and \cite[Lemma 6.5]{KL}.
\end{proof}

\begin{lemma}
\label{xxlem3.2}
Let $Y:=\bigoplus_{i=0}^{\infty} Y_i$ be an ${\mathbb N}$-graded 
domain. If $Z$ is a subalgebra of $Y$ containing $Y_0$ 
such that $\GKdim Z=\GKdim Y_0<\infty$, then $Z=Y_0$.
\end{lemma}

\begin{proof} Let $X$ denote the subalgebra $Y_0$.
Suppose $Z$ strictly contains $X$ as a subalgebra.
Since $Y$ is a graded algebra, $Z$ is an ${\mathbb N}$-filtered 
algebra with $F_0 Z=X$. By Lemma \ref{xxlem3.1}(3), $\GKdim Z\geq 
\GKdim \gr Z$. Since $\gr Z$ is an ${\mathbb N}$-graded 
sub-domain of $Y$ that strictly contains $X$ as the degree zero 
part of $\gr Z$, one can easily see that $\GKdim \gr Z\geq 
\GKdim (\gr Z)_0+1=\GKdim X+1$. Combining these inequalities, 
one obtains that $\GKdim Z\geq \GKdim X+1$. This contradicts 
the hypothesis that $\GKdim Z=\GKdim X$. Therefore $Z=X$.
\end{proof}

It is well-known that a  domain of finite GK-dimension is 
an Ore domain. Here is the main result of this section. 

\begin{theorem}
\label{xxthm3.3} 
Let $A$ be a finitely generated domain of finite GK-dimension.
Let $*$ be either blank, or $^H$, or $^I$. When 
$*$ is blank we further assume $A$ contains ${\mathbb Z}$. 
\begin{enumerate}
\item[(1)] 
If $A$ is strongly $\LND^*$-rigid, then $A$ is strongly 
cancellative.
\item[(2)] 
If $\ML^*(A[t])=A$, then $A$ is cancellative.
\end{enumerate}
\end{theorem}

\begin{proof} 
We prove (1) and note that the proof of (2) is similar.

Let $\phi: A[t_1,\ldots ,t_d]\to B[t_1,\ldots ,t_d]$ be an isomorphism
for some algebra $B$. By Lemma \ref{xxlem3.1}(2), 
$\GKdim B=\GKdim A<\infty$.  For each $i$, 
let $\partial_i:=\frac{\partial \;\;}{\partial t_i}$ when $*$ is blank 
and $\partial_i:=\{\Delta_i^n\}_{n=0}^{\infty}$ as in Example \ref{xxex2.4}  
when $*$ is either $^H$ or $^I$. 
We have that ${\rm ML}^*(B[t_1,\ldots ,t_d])$ is contained in $B$ since 
$\partial_i$, for $i=1,\cdots,d$, are locally nilpotent (higher) 
derivations of $B[t_1,\ldots ,t_d]$ and the intersection of the kernels 
of these maps is exactly $B$ (see Example \ref{xxex2.4}).  On the other 
hand, we have that ${\rm ML}^*(A[t_1,\ldots ,t_d])=A$ by hypothesis. 

If $\partial$ is a locally nilpotent 
derivation of $B[t_1,\ldots ,t_d]$ then $\phi^{-1}\circ \partial\circ \phi$ 
is a locally nilpotent derivation of $A[t_1,\ldots, t_d]$. Similarly 
if $\partial'$ is a locally nilpotent derivation of $A[t_1,\ldots ,t_d]$ 
then $\phi\circ \partial'\circ \phi^{-1}$ is a locally nilpotent 
derivation of $B[t_1,\ldots ,t_d]$.  Similarly, the higher derivations 
of $A[t_1,\ldots ,t_d]$ and $B[t_1,\ldots ,t_d]$ correspond.
Thus $\phi$ induces an algebra isomorphism 
$$\ML^*(A[t_1,\ldots,t_d])\cong \ML^*(B[t_1,\ldots, t_d]).$$
In particular $\phi$ maps $A$ into $B$. 
Let $Y=A[t_1,\ldots,t_d]$ with $\deg t_i=1$ and $Y_0=A$ and
$Z=\phi^{-1}(B)$. Then Lemma \ref{xxlem3.2} implies that 
$\phi^{-1}(B)=A$. So $A$ and $B$ are isomorphic.  The result follows.
\end{proof}

For the rest of this section we give some  corollaries.
We begin with a well-known result (see \cite[Lemma 3.2]{BS} 
or \cite[Lemma 2.1]{Ba} for related results). If $A$ is an Ore 
domain, let $Q(A)$ denote the fraction division ring of $A$.

\begin{lemma} 
\label{xxlem3.4}
Let $A$ be an Ore domain containing ${\mathbb Z}$. Suppose that $A$ 
is endowed with a nonzero locally nilpotent derivation $\delta$. 
Then the following hold.
\begin{enumerate}
\item[(1)]
$A$ is embedded in the Ore extension $E[x;\delta_0]$ and $E[x;\delta_0]$ 
is embedded in $Q(A)$, where $E=\{a\in Q(A)\mid \delta(a)=0\}$ and 
$\delta_0$ is a derivation of $E$.
\item[(2)]
$Q(A)=Q(E[x;\delta_0])$.
\item[(3)]
$\delta$ can be extended to a locally nilpotent derivation of 
$E[x;\delta_0]$ by declaring that $\delta(E)=0$ and $\delta(x)=1$.
\end{enumerate}
\end{lemma}

\begin{proof} (1)
Let $E$ denote the kernel of the unique extension of $\delta$ to 
$Q(A)$. Then $E$ is a division subalgebra of $Q(A)$.
Since $\delta$ is nonzero and locally nilpotent, we can 
find $x\in Q(A)\setminus E$ such that $\delta(x)\in E$. By replacing 
$x$ by $\alpha x$ for some $\alpha\in E$ we may assume that 
$\delta(x)=1$. Now for every $a\in E$ we have 
$\delta([x,a])=[\delta(x),a]=[1,a]=0$. Thus $[x,a]\in E$ for all $a\in E$.  
In particular, $[x,-]$ induces a derivation $\delta_0$ of $E$.  

Let 
$$W=\{ a\in Q(A)\mid \delta^n(a)=0, {\text{for some $n\geq 0$}}\}.$$
We claim that 
$W$ is a subset of  the subalgebra of $Q(A)$ generated by $E$ and $x$.  
Since $[x,E]\subseteq E$, we have that this subalgebra is just 
$$\sum_{i\ge 0} Ex^i.$$  To see the claim, we let $a\in W$.  Then there is 
some smallest $n$ for which $\delta^n(a)=0$.  We prove the claim by induction 
on $n$.  When $n=0$ we have $a\in E$ and so the result follows.   Now suppose 
that the claim holds whenever $\delta^j(a)=0$ for some $j<n$ and consider the 
case where $\delta^n(a)=0$ but $\delta^j(a)\neq 0$ for $j<n$.  Then 
$\delta^{n-1}(a)=\alpha\in E$ with $\alpha\neq 0$.  Since 
$\delta^{n-1}(\alpha x^{n-1}/(n-1)!) = \alpha$, we see that
$\delta^{n-1}(a-\alpha x^{n-1}/(n-1)!)=0$ and so by the induction hypothesis 
$a\in \sum Ex^i$.  The claim follows.  

It is clear that $\sum Ex^i\subseteq W$. So $W=\sum Ex^i$. Since $\delta$ 
is in $\LND(A)$, $A\subset W$. Thus $A$ embeds in the subalgebra $W$ 
generated by $E$ and $x$. Since 
$[x,\alpha]=\delta_0(\alpha)$ for $\alpha\in E$, we see that $W$ is isomorphic 
to a homomorphic image of $E[t;\delta_0]$.  We claim that $W$ cannot be 
isomorphic to a 
proper homomorphic image of $E[t;\delta_0]$.  To see this, note that if it 
were $x$ would satisfy a non-trivial equation
$$x^d + \beta_{d-1} x^{d-1}+ \cdots + \beta_0=0$$ 
for some $d\ge 1$ and $\beta_{d-1},\ldots ,\beta_0\in E$.  We may assume 
without loss of generality that $d$ is minimal.  Then applying $\delta$ and 
using the fact that $\delta(x)=1$ and that $\delta$ is zero on $E$ gives
$$x^{d-1}+\sum_{j=1}^{d-1} jd^{-1} \beta_j x^{j-1}=0,$$ 
contradicting the minimality of $d$. Thus we see that $A$ embeds 
in $W$ which is isomorphic to $E[x;\delta_0]$ as required. 

Both (2) and (3) are clear.
\end{proof}

The following result was proved in \cite{Ma2} in the commutative case. 

\begin{lemma} 
\label{xxlem3.5}
Let $A$ be a finitely generated Ore domain over $k$ that 
contains ${\mathbb Z}$. If $A$ is $\LND$-rigid, then $\ML(A[x])=A$. 
\end{lemma}

\begin{proof} 
Let $C={\rm ML}(A[x])$.  Note that $C\subseteq A$ since differentiation 
with respect to $x$ gives a locally nilpotent derivation of $A[x]$ and the 
kernel of this map is exactly $A$.  It suffices to show that $C\supseteq A$.  
Suppose that there is a locally nilpotent derivation $\delta$ of $A[x]$ that 
does not send $A$ to zero. 
Suppose $a_1,\ldots ,a_s$ generate $A$ as a $k$-algebra. 
Then $\delta(A)\subseteq A\delta(a_1)A+\cdots A \delta(a_s)A$ and so there 
exists some smallest $m\ge 0$ such that 
$\delta(A)\subseteq A+Ax+\cdots +Ax^m$.  If $m=0$, then $\delta(A)
\subseteq A$. Since $A$ is LND-rigid, $\delta(A)=0$. This
yields a contradiction and therefore $m\geq 1$. We write 
$$\delta(a)=\mu(a) x^m +{\rm lower~degree~terms}$$ 
for some derivation $\mu$ of $A$.
We now consider the following three cases.  
\vskip 2mm
\emph{Case I:} $\delta(x)\in A+Ax+\cdots +Ax^{m}$.
\vskip 2mm
In this case we have $\delta(x^{i})\subseteq \sum_{n=0}^{i+m-1} Ax^n$
and  $\delta(Ax^i)\subseteq \sum_{n=0}^{i+m}Ax^n$ for all $i$. 
Thus for every $a\in A$ we have 
$$\delta^2(a)=\mu^2(a)x^{2m}+{\rm lower~degree~terms}.$$  
More generally, we see that 
$$\delta^j(a)=\mu^j(a) x^{mj}+{\rm lower~degree~terms}.$$  
Thus $\mu$ is a locally nilpotent derivation and so $\mu(A)=0$, contradicting 
the minimality of $m$.  Thus $\delta(A)=0$ in this case.
\vskip 2mm
\emph{Case II:} $\delta(x)=bx^{m+1}+{\rm lower~degree~terms}$ for 
some $b\neq 0$ in $A$.
\vskip 2mm
Applying $\delta$ to the equation $[x,a]=0$, one sees that $b$
commutes with every $a$ in $A$, and so $b$ is in the center of $A$.
Now we define a new derivation $\delta'$ of $A[x]$ by declaring that 
$\delta'(a)=\mu(a)x^m$ for $a\in A$ and $\delta'(x)=bx^{m+1}$.  
Then we see that $\delta'$ sends $Ax^i$ to $Ax^{i+m}$ for every 
$i\ge 0$. We can view $\delta'$ as an associated graded derivation
of $\delta$. Since $\delta$ is locally nilpotent, $\delta'$ 
is a locally nilpotent derivation of $A[x]$ \cite[Lemma 4.11]{CPWZ2}. 
Applying Lemma \ref{xxlem3.4} to the algebra $A[x]$, $A[x]$ embeds 
in $E[y;\delta_0]$ where $\delta_0$ is a derivation of $E$.
Moreover, $\delta'$ extends to a locally nilpotent derivation of 
$E[y;\delta_0]$ by declaring that $\delta'(E)=0$ and $\delta'(y)=1$.  
Under this embedding $x=p(y)$ for some nonzero polynomial $p$. 
Let $d$ denote the degree of this polynomial.  
Then $bx^{m+1}$ gets sent to $q(y)p(y)^{m+1}$ for some nonzero 
polynomial $q(y)$.  
But since $\delta'(x)$ is nonzero, it has degree exactly $d-1$ and so we 
have $(m+1)d+\deg q(y)=d-1$, which is impossible.  
\vskip 2mm
\emph{Case III:} $\delta(x)=bx^{i}+{\rm lower~degree~terms}$ for 
some $b\neq 0$ in $A$ and some $i>m+1$.
\vskip 2mm
In this case we see that, for each $n\geq 2$, 
$$\delta^n(x)=\left\{ \prod_{s=1}^{n-1} ((i-1)s+1)\right\}  \; b^n x^{(i-1)n+1}
+{\rm lower~degree~terms},$$ 
so $\delta$ cannot be locally nilpotent, which contradicts the hypothesis.  

Combining these cases, we see that $\delta(A)=0$. The result follows.
\end{proof}

We next give the proof of Theorem \ref{xxthm0.4}.

\begin{theorem}
\label{xxthm3.6} 
Let $A$ be a finitely generated domain containing ${\mathbb Z}$ and 
suppose that $A$ has finite GK-dimension. If $A$ is $\LND$-rigid, then $A$ 
is cancellative.
\end{theorem}

\begin{proof} Since $A$ is a domain of finite GK-dimension, it is 
an Ore domain. By Lemma \ref{xxlem3.5}, $\ML(A[x])=A$. The assertion
follows from Theorem \ref{xxthm3.3}(2).
\end{proof}

We now prove Theorem \ref{xxthm0.5}. We say an algebra $A$ is {\it PI} 
if it satisfies a polynomial identity. 

\begin{corollary} 
\label{xxcor3.7}
Let $A$ be a domain of GK-dimension two over an 
algebraically closed field $k$ of characteristic zero. 
\begin{enumerate}
\item[(1)]
If $A$ is PI and not commutative, then $A$ is $\LND$-rigid. As 
a consequence, if we assume in addition that $A$ is finitely 
generated over $k$, then $A$ is cancellative.
\item[(2)]
If $A$ is not PI, then $A$ is universally cancellative.
\end{enumerate}
\end{corollary}

\begin{proof} 
(1) If $A$ is not $\LND$-rigid, then there is a nonzero locally 
nilpotent derivation $\delta$ of $A$. So the kernel of $\delta$ 
is not equal to $A$.  As in 
Lemma \ref{xxlem3.4} let $E$ denote the set of elements $a\in Q(A)$ 
such that $\delta(a)=0$.  By Lemma \ref{xxlem3.4}, $A$ embeds in 
$W:=E[x;\delta_0]$ for some derivation $\delta_0$ of $E$.   Since
$W$ is a subalgebra of $Q(A)$, $Q(A)$ is infinite-dimensional as a 
left and right $E$-vector space.  Hence $E$ has GK-dimension one 
\cite[Theorem 1.3]{Be}. 
Since $E$ is a subalgebra of $Q(A)$, it is PI and so by Tsen's theorem, 
$E$ is commutative, whence $E$ is a 
field.  By Lemma \ref{xxlem3.4}(3), $W:=E[x;\delta_0]$ is 
a subring of $Q(A)$. Since $A$ is  PI, $Q(A)$ is also PI and hence $E[x;\delta_0]$
is PI. We observe that this gives $\delta_0=0$.  To see this, suppose 
that there is some $\alpha \in E$ such that $\beta:=\delta_0(\alpha)\neq 0$.
Then $[\beta^{-1}x,\alpha]=\beta^{-1}\delta_0(\alpha)=1$ and so 
in this case we would have that $E[x;\delta_0]$ contains a copy of 
the Weyl algebra over $\mathbb{Q}$, which contradicts the fact 
that $E[x;\delta_0]$ is PI.  
Thus $\delta_0=0$ and $W$ is commutative. So 
$A$ is commutative, yielding a contradiction. The result follows.

The consequence follows from the main assertion and
Theorem \ref{xxthm3.6}.

(2) If $A$ is not PI and has GK-dimension two, then, by
\cite[Corollary 2]{SZ}, $C(A)=k$.
The assertion now follows from Proposition \ref{xxpro1.3}.
\end{proof}

\begin{definition}
\label{xxdef3.8}
An Ore domain $A$ is called {\it birationally affine-ruled} 
if $Q(A)=D(x)$ for some division algebra $D$ and 
{\it birationally Weyl-ruled} 
if $Q(A)=Q(E[x;\delta_0])$ for some division algebra $E$ 
and some nonzero derivation 
$\delta_0$ of $E$.
\end{definition}

By Lemma \ref{xxlem3.4}, if $A$ has a nonzero locally 
nilpotent derivation, then $A$ is either birationally affine-ruled
or birationally Weyl-ruled.

\begin{corollary}
\label{xxcor3.9}
Let $A$ be a finitely generated PI domain containing ${\mathbb Z}$ with 
finite GK-dimension. If $A$ is not birationally affine-ruled, then 
$A$ is $\LND$-rigid and cancellative.
\end{corollary}

\begin{proof} By Theorem \ref{xxthm3.6}, it suffices to show that
$A$ is $\LND$-rigid.

If $A$ is not $\LND$-rigid, then $A$ is endowed with a 
nonzero locally nilpotent derivation. By Lemma \ref{xxlem3.4}, 
$A\subset E[x;\delta_0]\subset Q(A)$, where $E$ is a division
subring of $Q(A)$. Since $A$ is PI, so are $Q(A)$ and $E[x;\delta_0]$. 
Then the center of $E[x;\delta_0]$ is not a subring of $E$. 
Let $f=a_nx^n+a_{n-1}x^{n-1}+\cdots +a_0$ be a central element 
in $E[x;\delta_0]$ for some $n\geq 1$ and $a_0\neq 0$.
Since $f$ is central,
$$0=[x,f]=\sum_{i=0}^n [x,a_i] x^i=\sum_{i=0}^n \delta_0(a_i) x^i,$$
implying that $\delta_0(a_i)=0$ for all $i$. For every $e\in E$,
$$0=[e,f]=[e,a_{n}] x^n+{\text{lower~degree~terms}},$$
which implies that $[e,a_n]=0$. Hence, $a_n$ is in the center
of $E[x;\delta_0]$. By replacing $f$ by $a_{n}^{-1} f$, we may 
assume that $a_n=1$. A straightforward calculation gives
$$\begin{aligned}
0&=[e,f]=ex^n-(ex^n+n\delta_0(e) x^{n-1}+{\text{lower~degree~terms}})\\
&\qquad\qquad\qquad +[e,a_{n-1}]x^{n-1}+{\text{lower~degree~terms}}\\
&=(-n\delta_0(e)+[e,a_{n-1}])x^{n-1}+{\text{lower~degree~terms}}.
\end{aligned}
$$
Hence $-n\delta_0(e)+[e,a_{n-1}]=0$ or $\delta_0(e)=[e, b]$
where $b=\frac{1}{n}a_{n-1}$. Then $E[x,\delta_0]=E[x']$ 
where $x'=x+b$.
So $A$ is birationally affine-ruled, a contradiction.
\end{proof}


\section{Discriminant}
\label{xxsec4}

We recall the definition of the discriminant in the noncommutative
setting and everything in this section is taken from \cite{CPWZ1, CPWZ2}.
Let $R$ be a commutative algebra and let $B$ and $F$ be algebras 
both of which contain $R$ as a subalgebra. In our applications, $F$ 
will either be $R$ or a ring of fractions of $R$. An $R$-linear 
map $\tr: B\to F$ is called a \emph{trace map} if $\tr(ab)=\tr(ba)$ 
for all $a,b\in B$.

If $B$ is the $w\times w$-matrix algebra $M_w(R)$ over $R$, the 
internal trace $\tr_{\INT}: B\to R$ is defined
to be the usual matrix trace, namely, $\tr_{\INT}((r_{ij}))=
\sum_{i=1}^w r_{ii}$.
Let $B$ be an $R$-algebra, and suppose that $B_F:=B\otimes_R F$ is 
finitely generated and free over $F$, where $F$ is a localization of
$R$. Then left
multiplication defines a natural embedding of $R$-algebras $lm:B\to
B_F\to \End_F(B_F)\cong M_w(F)$, where $w$ is the rank $\rk(B_F/F)$. 
Then we have a \emph{regular trace},
by composing: 
$$\tr_{\textup{reg}}: B\xrightarrow{lm} M_w(F)\xrightarrow{\tr_{\INT}} F.$$
Usually we use the regular trace even if other trace functions 
exist. The following definition is well-known, see Reiner's book \cite{Re}.
For an algebra $A$, let $A^\times$ denote the set of invertible elements in $A$.
If $f,g\in A$ and $f=cg$ for some $c\in A^\times$, then we write
$f=_{A^\times} g$.

\begin{definition} \cite[Definition 1.3]{CPWZ1}
\label{xxdef4.1} Let $\tr: B\to F$ be a trace map and $v$ be a fixed
integer. Let $Z:=\{z_i\}_{i=1}^v$ be a subset of $B$.
\begin{enumerate}
\item[(1)]
The \emph{discriminant} of $Z$ is defined to be
\[
d_v(Z:\tr)=\det(\tr(z_iz_j))_{v\times v}\in F.
\]
\item[(2)] 
\cite[Section 10, p.~126]{Re}.
The \emph{$v$-discriminant ideal} (or \emph{$v$-discriminant
$R$-module}) $D_v(B,\tr)$ is the $R$-submodule of $F$ generated by
the set of elements $d_v(Z:\tr)$ for all $Z=\{z_i\}_{i=1}^v\subset
B$.
\item[(3)]
Suppose $B$ is an $R$-algebra which is finitely generated free over
$R$ of rank $w$. If $Z$ is an $R$-basis of $B$, the \emph{discriminant} of $B$ 
over $R$ is defined to be
\[
d(B/R)=_{R^\times} d_w(Z:\tr).
\]
Note that $d(B/R)$ is well-defined up to a scalar in
$R^\times$ \cite[p.66, Exer 4.13]{Re}.
\end{enumerate}
\end{definition}

We refer to the books \cite{AW, Re, St} for the classical definition
of the discriminant and its connection with the above definition. 

To cover a larger class of algebras that are not free over their 
centers, we need a modified version of the discriminant. Let $B$ be a domain
and let ${\mathcal D}:=\{d_i\}_{i\in I}$ be a set of elements in $B$. 
A normal element $x\in B$ is called a \emph{common divisor} if 
$d_i=d_i' x$ for some $d'_i$ for all $i\in I$. We say a normal element 
$x\in B$ is the \emph{greatest common divisor} or \emph{gcd} of 
${\mathcal D}$, denoted by $\gcd {\mathcal D}$, if 
\begin{enumerate}
\item
$x$ is a common divisor of ${\mathcal D}$, and
\item
for every common divisor $y$ of ${\mathcal D}$, $x=cy$ for some 
$c\in B$.
\end{enumerate}

It follows from part (b) that the gcd of any subset ${\mathcal D}
\subseteq B$ (if it exists) is unique up to a scalar in $B^\times$.

\begin{definition}
\cite[Definition 1.2]{CPWZ2} 
\label{xxdef4.2} Let $\tr: B\to R$ be a trace map and let $v$ be a 
positive integer. Let $Z$ {\rm{(}}respectively, $Z'${\rm{)}} denote 
a $v$-element subset $\{z_i\}_{i=1}^v$ 
{\rm{(}}respectively, $\{z'_i\}_{i=1}^v${\rm{)}} of $B$.
\begin{enumerate}
\item[(1)]
The \emph{discriminant} of the pair $(Z,Z')$ is defined to be
\[
d_v(Z,Z':\tr)=\det(\tr(z_iz'_j))_{v\times v}\in R.
\]
\item[(2)]
The \emph{modified $v$-discriminant ideal}  $MD_v(B,\tr)$ is 
the ideal of $R$ generated by
the set of elements $d_v(Z,Z':\tr)$ for all $Z, Z' \subset B$.
\item[(3)]
The \emph{$v$-discriminant} $d_v(B/R)$ is defined to be the gcd 
in $B$ (possibly not in $R$) 
of elements $d_v(Z,Z':\tr)$ for all $Z, Z' \subset B$. Equivalently,
the $v$-discriminant $d_v(B/R)$ is the gcd 
in $B$ of all elements in $MD_v(B,\tr)$.
\end{enumerate}
\end{definition}

In Definition \ref{xxdef4.2}(3), we are taking the gcd in $B$, not in $R$.
If $d_v(B/R)$ exists, then the ideal $(d_v(B/R))$ of $B$ generated 
by $d_v(B/R)$ is the smallest
principal ideal of $B$ that contains $MD_v(B:\tr)B$. 
If $B$ is an $R$-algebra which is finitely generated 
free over $R$ and if $w=\rk(B/R)$, then  $MD_w(B:\tr)$  equals 
$D_w(B:\tr)$, both of which  are generated by
a single element $d(B/R)$. In this case it is also true that $d(B/R)
=_{B^\times}d_w(B/R)$. 
If $v>\rk(B/R)$, then $d_v(B/R)=0$ \cite[Lemma 1.9(2)]{CPWZ2}. 

Some explicit examples of discriminants are given in 
\cite{CPWZ1,CPWZ2, CYZ1, CYZ2}.

The next lemma is straightforward by using commutative 
algebra argument. 

\begin{lemma}
\label{xxlem4.3} 
Let $A$ and $B$ be PI domains with centers $C_A$ 
and $C_B$ respectively. Let $\tr: A\to Q(C_A))$ be a 
$C_A$-linear trace function. Let $R$ be a $k$-flat commutative 
algebra such that $A\otimes R$ is a domain.    
\begin{enumerate}
\item[(1)]
If $\tr$ is the regular trace, then 
$\tr\otimes R: A\otimes R\to Q(C_A) \otimes R$ is 
the regular trace. 
\item[(2)]
Suppose that $R$ is $k$-free.
Then the image of $\tr$ is in $C_A$ if and only if the image of
$\tr\otimes R$ is in $C_A\otimes R$.  
\item[(3)]
Suppose that $\phi: A\to B$ is an algebra isomorphism. Then 
$\phi\circ \tr\circ \phi^{-1}: B\to Q(C_B)$ is the regular trace
if and only if $\tr$ is the regular.
\end{enumerate}
\end{lemma}

One of our key lemmas is the following, which suggests that the
discriminant controls the group of automorphisms.

\begin{lemma}
\label{xxlem4.4} Let $\phi: A\to B$ be an isomorphism of algebras.
Let $C_A$ and $C_B$ be the center of $A$ and $B$ respectively.
Suppose that $tr_A$ {\rm{(}}respectively, $tr_B${\rm{)}}
is the regular trace $A\to C_A$ {\rm{(}}resp. $B\to C_B${\rm{)}} and 
that the image of $tr_A$ is in $C_A$ {\rm{(}}resp. the image of 
$tr_B$ is in $C_B${\rm{)}}. Let $w$ be a positive integer. 
Then the following hold:
\begin{enumerate} 
\item[(1)]
$\phi$ maps the discriminant ideal $D_w(A,\tr_A)$ to $D_w(B,\tr_B)$;
\item[(2)]
if $A$ is a finitely generated free module over $C_A$, then 
$\phi(d(A/C_A))=_{C_B^\times} d(B/C_B)$;
\item[(3)]
$\phi$ maps the modified discriminant ideal $MD_w(A,\tr_A)$ to 
$MD_w(B,\tr_B)$;
\item[(4)]
$\phi$ maps the $w$-discriminant $d_w(A/C_A)$ to 
$d_w(B/C_B)$.
\end{enumerate}
\end{lemma}

\begin{proof} By Lemma \ref{xxlem4.3}(3),
$\phi (tr_A(x))=tr_B(\phi(x))$ for all $x\in A$. The rest follows from 
this observation.
\end{proof}

The concept of a dominating element was introduced in
\cite{CPWZ1, CPWZ2} to handle the noncommutative {\bf AP}.
We now recall this notion. 

Let $R$ be an algebra over $k$. We say $R$ is \emph{connected graded} if
$R=k\oplus R_1\oplus R_2\oplus \cdots$ and $R$ is \emph{locally finite}
if each $R_i$ is finitely generated over $k$. We now consider 
filtered rings $A$. Let $Y$ be a finitely generated free $k$-submodule of 
$A$ such that $k1_A \cap Y=\{0\}$. Consider the \emph{standard filtration} 
defined by $F_n A:= (k1_A +Y)^n$ for all $n\geq 0$. Assume that this 
filtration is exhaustive and that the associated graded
ring $\gr A$ is connected graded (or the map $k\to A$ is 
injective). For each element $f\in
F_n A\setminus F_{n-1} A$, the associated element in $\gr A$ is
defined to be $\gr f=f+F_{n-1} A\in (\gr_F A)_n$. The degree of an
element $f\in A$, denoted by $\deg f$, is defined to be the degree 
of $\gr f$.

If $\gr A$ is a domain, then, for any elements $f_1, f_2 \in A$, 
\begin{equation}\notag
\deg (f_1 f_2) = \deg f_1 + \deg f_2.
\end{equation}

If $\gr A$ is a connected graded domain, it is easy to see that
$A^\times =k^\times$.  As usual, we assume that
$k\subseteq A$. In this case $\gr A$ is connected graded. 
If $R$ is a subalgebra of $A$, then
$R^\times \subseteq A^\times =k^\times$.

\begin{definition}\cite[Definition 2.1(2)]{CPWZ1}
\label{xxdef4.5} Retain the above notation.
Suppose that $Y=\bigoplus_{i=1}^n kx_i$ generates
$A$ as an algebra.
Assume that $\gr A$ is a connected graded domain.
An element $f\in A$ is called \emph{dominating} if, for every 
testing ${\mathbb N}$-filtered PI algebra $T$ with $\gr T$ being a 
connected graded domain, and for every testing subset 
$\{y_1,\ldots,y_{n}\}\subset T$ that is linearly independent in
the quotient $k$-module $T/F_0 T$, there is a presentation
of $f$ of the form $f(x_1,\ldots,x_{n})$ in the free algebra $k\langle 
x_1,\ldots,x_{n}\rangle$, such that the following hold: either  
$f(y_1,\ldots,y_{n})=0$, or 
\begin{enumerate}
\item
$\deg f(y_1,\ldots,y_{n})\geq \deg f$, and
\item
$\deg f(y_1,\ldots,y_{n})> \deg f$ if, further, 
$\deg y_{i_0}>1$ for some $i_0$.
\end{enumerate}
\end{definition}

Suppose now $A$ is generated by elements in $Y=\bigoplus_{i=1}^n kx_i$
of degree 1. A monomial $x_1^{b_1}\cdots x_n^{b_n}$ is said to have
degree \emph{component-wise less than} (or, \emph{cwlt}, for short)
$x_1^{a_1}\cdots x_n^{a_n}$ if $b_i\leq a_i$ for all $i$ and
$b_{i_0}<a_{i_0}$ for some $i_0$.  We write $f=cx_1^{b_1}\cdots
x_n^{b_n}+\cwlt$ if $f-cx_1^{b_1}\cdots x_n^{b_n}$ is a linear
combination of monomials with degree component-wise less than
$x_1^{b_1}\cdots x_n^{b_n}$.  
If $f=x_1^{b_1}\cdots x_{n}^{b_{n}}+\cwlt$ for some 
$b_1,\cdots, b_{n}\geq 1$, then $f$ is dominating (see the proof of 
\cite[Lemma 2.2]{CPWZ1}). 
In the next section we will
introduce a notion of effectiveness to deal with noncommutative
{\bf ZCP}.

The next result is a key lemma. Let $R$ be a commutative algebra. We say 
that $A\otimes R$ is $A$-closed if, for every $0\neq f\in A$ and 
$x,y\in A\otimes R$, the equation $xy=f$ implies that $x,y\in A$ up 
to units of $A\otimes R$. For example, if $R$ is connected graded 
and $A\otimes R$ is a domain, then $A\otimes R$ is $A$-closed. 

\begin{lemma}\cite[Lemma 1.12]{CPWZ2}
\label{xxlem4.6}
Let $\tr: A\to C$ be a $C$-linear trace function where $C$ is a central 
subalgebra of $A$. 
Let $R$ be a $k$-flat commutative algebra such that $A\otimes R$ is a 
domain and $v$ be a positive integer. 
\begin{enumerate}
\item[(1)]
$MD_v(A\otimes R:\tr\otimes R)=MD_v(A:\tr)\otimes R$. 
\item[(2)]
Suppose $A\otimes R$ is $A$-closed.
If $d_v(A/C)$ exists, then $d_v(A\otimes R/C\otimes R)$ exists and 
is equal to $d_v(A/C)$. 
\end{enumerate}
\end{lemma}

Now we are ready to state the main result of this section, 
which is basically \cite[Lemma 3.3(3)]{CPWZ1}. The proof
is given in the next section.

\begin{theorem}
\label{xxthm4.7} Let $A$ be a PI algebra. 
Suppose that the $w$-discriminant $d_w(A/C)$ is dominating 
for some $w$. Then the following hold.
\begin{enumerate}
\item[(1)]
$A$ is strongly $\LND^H$-rigid.
\item[(2)]
If $A$ has finite GK-dimension, then $A$ is strongly cancellative.
\end{enumerate}
\end{theorem}

The above theorem applies to many algebras including ones
listed below.

\begin{example}
\label{xxex4.8}
It is known that the following algebras have dominating discriminants
\cite{CPWZ1}.
\begin{enumerate}
\item[(1)]
$k_q[x_1,\ldots, x_{n}]$ where $n$ is an even 
number and $1\neq q$ is a root of unity.
\item[(2)]
$k\langle x, y\rangle/(x^2y-yx^2, y^2x+xy^2)$. 
\item[(3)] 
$k\langle x, y\rangle/(yx-q xy-1)$  where $1\neq q$ is a root of unity.
\item[(4)] 
finite tensor products of algebras of the form (1),(2),(3) above
\cite[Lemma 5.4]{CPWZ1}. 
\end{enumerate}
By Theorem \ref{xxthm4.7}(2), 
these algebras  are strongly cancellative.
\end{example}

\section{Effectiveness controls cancellation}
\label{xxsec5}

First we introduce the notion of effectiveness that plays an 
important role in the resolution of the noncommutative {\bf ZCP}.

\begin{definition}
\label{xxdef5.1} 
Let $A$ be a domain and suppose that $Y=\bigoplus_{i=1}^n kx_i$ 
generates $A$ as an algebra.
An element $f\in A$ is called \emph{effective} if the following
conditions hold.
\begin{enumerate}
\item[(1)]
There is an ${\mathbb N}$-filtration $\{F_i A\}_{\geq i}$ on $A$ 
such that the associated graded ring $\gr A$ is a domain 
(one possible filtration is the trivial filtration $F_0 A=A$).
With this filtration we define the degree of elements in $A$, denoted
by $\deg_A$.
\item[(2)]
For every testing ${\mathbb N}$-filtered PI algebra $T$ with 
$\gr T$ being an ${\mathbb N}$-graded domain and for every 
testing subset $\{y_1,\ldots,y_{n}\}\subset T$ satisfying 
\begin{enumerate}
\item[(a)] 
it is linearly independent in the quotient $k$-module $T/k1_T$, and 
\item[(b)] 
$\deg y_i\geq \deg x_i$ for all $i$ and $\deg y_{i_0}>\deg x_{i_0}$ 
for some $i_0$,
\end{enumerate}
there is a presentation of $f$ of the form 
$f(x_1,\ldots,x_{n})$ in the free algebra $k\langle x_1,\ldots,x_{n}
\rangle$, such that either $f(y_1,\ldots,y_n)$ is zero or 
$\deg_{T} f(y_1,\ldots,y_n)>\deg_A f$.
\end{enumerate}
\end{definition}

Note that the definition of a dominating element 
[Definition \ref{xxdef4.5}] is slightly different from the 
definiton of effectiveness. For example, we do not require 
Definition \ref{xxdef4.5}(a) in the definition of effectiveness.
On the other hand, one only needs to test those $T$ such that 
$\gr T$ is connected graded in the definition of a dominating element.
It is easy to check that elements $f:=x_1^{b_1}\cdots x_n^{b_n}+\cwlt$
is effective. We have already seen that there are many examples (those example 
given in \cite{CPWZ1,CPWZ2}) of noncommutative algebras whose 
discriminant is dominating and of the form 
$x_1^{b_1}\cdots x_n^{b_n}+\cwlt$, whence effective. Here is the main result 
in this section, which is a slight generalization of Theorem \ref{xxthm4.7}.

\begin{theorem}
\label{xxthm5.2} Let $A$ be a PI domain such that the $w$-discriminant 
over its center is effective {\rm{(}}or dominating in part {\rm{(2))}} 
for some $w$. 
\begin{enumerate}
\item[(1)]
Suppose $A$ has finite GK-dimension. Let $R$ be an affine 
$k$-free connected graded commutative domain such that $A\otimes R$ 
is a domain.  If $A\otimes R\cong B\otimes R$ for some algebra $B$, 
then $A\cong B$.  As a consequence, $A$ is strongly cancellative.
\item[(2)]
$A$ is strongly $\LND^H$-rigid. 
\end{enumerate}
\end{theorem}

\begin{proof}
(1)  
Let $\phi$ be the isomorphism from $A\otimes R$ to 
$B\otimes R$. By Lemma \ref{xxlem3.1}(2),
$$\GKdim B=\GKdim (B\otimes R)-\dim R=\GKdim (A\otimes R)-\GKdim R=
\GKdim A<\infty.$$
Let $C_A$ and $C_B$ be the center of $A$ and $B$ respectively. 
Since $R$ is $k$-free, the center of $A\otimes R$ and $B\otimes R$ 
are $C_A\otimes R$ and $C_B\otimes R$ respectively. By Lemma \ref{xxlem4.3},
the regular trace functions for different algebras are all compatible.
By Lemma \ref{xxlem4.4}(4),  $\phi$ maps $d_w(A\otimes R/C_A\otimes R)$ to 
$d_w(B\otimes R/C_B\otimes R)$. By Lemma \ref{xxlem4.6}(2),

$$d_w(A\otimes R/C_A\otimes R)=d_w(A/C_A), \quad {\text{and}}
\quad  d_w(B\otimes R/C_B\otimes R)=d_w(B/C_B).$$ 
This implies that $\phi(d_w(A/C_A))=d_w(B/C_B)\in B$. 

Let $f$ denote $d_w(A/C_A)$. By hypothesis, $f$ is effective.
Suppose $A$ is generated by $Y=\bigoplus_{i=1}^n kx_i$ 
as a $k$-algebra, and write $f$ as $f(x_1,\cdots,x_d)$ as in 
Definition \ref{xxdef5.1}. We take the testing algebra to be 
$T=B\otimes R$. Since $T$ is a domain (since $T\cong A\otimes R$),
and $R$ is connected graded, $T$ is an ${\mathbb N}$-graded domain
by setting $\deg b=0$ for all $b\in B$ and $\deg r=\deg_R r$ for
all homogeneous element $r\in R$. In particular, $T$ is an 
${\mathbb N}$-filtered algebra with $F_0 T=B$ such that $\gr T$ 
is a domain. Now take a testing subset $\{y_1,\cdots, y_d\}
\subset T$ by setting $y_i=\phi(x_i)\in T$ for $i=1,\cdots, d$. 
We claim that $y_i\in B$ for all $i$. If not, there is some $i_0$ 
such that $y_{i_0}$ is not in $B=F_0 T$. By the effectiveness of 
of $f$, $f(y_1,\cdots, y_d)$ is either zero or  not in $B:=F_0 T$. 
However, 
$$f(y_1,\cdots, y_d)=f(\phi(x_1),\cdots, \phi(x_d))=
\phi(f(x_1,\cdots, x_d))=\phi(f).$$
By the last statement in the previous paragraph, 
$\phi(f)=d_w(B/C_B)$ is a nonzero element in $B$, a contradiction. 
Therefore each $y_i\in B$.
This means that $\phi$ maps $x_i$ to $y_i$ in $B$. Since $A$ is
generated by $x_i$, the image of $A$ under $\phi$ is a subalgebra
of $B$. So $\phi^{-1}(B)$ is a subalgebra $A\otimes R$ 
that contains $A$ as a subalgebra. Note that $\GKdim \phi^{-1}(B)
=\GKdim B=\GKdim A$, by the first paragraph of the proof. 
By Lemma \ref{xxlem3.2}, $\phi^{-1}(B)=A$.
Therefore the image of $A$ under $\phi$ is exactly $B$, which 
implies that $\phi: A\cong B$. 

(2) Since the proofs for the ``effective'' case and the ``dominating''
case are very similar, we combine two proofs together.

Suppose $A$ is generated by $\{x_1,\ldots,x_n\}$ as in
Definition \ref{xxdef5.1} (or Definition \ref{xxdef4.5}). 
Let $R=k[t_1,\ldots,t_d][t]$. By Lemma \ref{xxlem4.6}(2),
$$d_w(A\otimes  R/C_A\otimes R)=d_w(A/C_A)=:f,$$ 
which is effective (or dominating) by hypothesis.
Let $\partial\in \LND^H(A[t_1,\cdots,t_d])$. By definition,
$G:=G_{\partial, t}\in \Aut_{k[t]}(A[t_1,\cdots,t_d][t])$.
For each $j$, 
$$G(x_j)=x_j+\sum_{i\geq 1} t^i \partial_i(x_j).$$
We take the test algebra $T$ to be $A[t_1,\cdots,t_d][t]$ where 
the filtration on $T$ is induced by the filtration on $A$ together
with $\deg t_s=1$ for all $s=1,\ldots,d$ and $\deg t=\alpha$
where $\alpha$ is larger than  
$\deg \partial_i(x_j)$ for all $j=1,\ldots,n$ and all $i\geq 1$. 
(In the dominating case, $\gr T$ is a connected graded domain.)
Now set $y_j=G(x_j)\in T$. By the choice of $\alpha$, 
we have that
\begin{enumerate}
\item[(a)]
$\deg y_j\geq \deg x_j$, and that
\item[(b)]
$\deg y_j=\deg x_j$ if and only if $y_j=x_j$. 
\end{enumerate}
If $G(x_j)\neq x_j$ for some $j$,
by effectiveness  as in Definition \ref{xxdef5.1}
(or dominating as in Definition \ref{xxdef4.5}),
$\deg f(y_1,\cdots,y_n)>\deg f$. So $f(y_1,\cdots,y_n)
\neq_{A^\times} f$. But $f(y_1,\cdots,y_n)=G(f)=_{A^\times} f$ 
by Lemma \ref{xxlem4.4}(4), a contradiction. Therefore
$G(x_j)=x_j$ for all $j$. As a consequence,
$\partial_i(x_j)=0$ for all $i$, or $x_j\in \ker \partial$. 
Since $A$ is generated by $x_j$'s, $A\subset
\ker\partial$. Thus $A\subseteq \ML^H(A[t_1,\cdots,t_d])$.
It is clear that $A\supseteq \ML^H(A[t_1,\cdots,t_d])$ 
[Example \ref{xxex2.4}], 
so the assertion follows.
\end{proof}

Part (1) of the above theorem shows that $A$ is close to be 
universally cancellative. Part (2) is 
Theorem \ref{xxthm4.7} is a special case of Theorem \ref{xxthm5.2}. 
We are now ready to show Theorem \ref{xxthm0.6}.

\begin{proof}[Proof of Theorem \ref{xxthm0.6}]
Since $A$ is finitely generated over its affine center,
$A$ has finite GK-dimension. The assertion follows immediately 
from Theorem \ref{xxthm5.2}.
\end{proof}

Next we consider some examples studied in \cite{CPWZ1, CPWZ2}.
Effectiveness of an element is easy to check sometimes. The
following lemma is easy.

\begin{lemma}
\label{xxlem5.3}
Suppose $A$ is generated by $\{x_1,\ldots,x_n\}$ as in 
Definition {\text{\ref{xxdef5.1}}}. 
\begin{enumerate}
\item[(1)]
$f=g_0 x_1 g_1 x_2 \cdots x_{n-1}g_{n-1} x_n g_n$ is effective
if $g_i\in A$ are nonzero.
\item[(2)]
If $f=x_1^{b_1}\cdots, x_n^{b_n}$, then $f$ is effective if and only if
$b_i\ge 1$ for all $i$. 
\item[(3)]
If $f$ is effective, and $g\neq 0$, then $fg$ and $gf$ are effective.
\item[(4)]
Suppose that $A$ is generated by two subalgebras $A_1$ and $A_2$.
If $f_1$ and $f_2$ are effective elements in $A_1$ and $A_2$ respectively,
then $f_1g f_2$ is an effective element in $A$ for every nonzero $g
\in A$. 
\item[(5)]
If $A$ is generated by $\{x_1,x_2\}$ and $g,h\in A$ and $f=g (x_1x_2+a x_2x_1)=
h(x_1x_2+b x_2 x_1)$ for some scalars $a\neq b$. Then $f$ is effective.
\item[(6)]
If $b_1,\ldots,b_n$ are positive integers and 
$f= x_1^{b_1}x_2^{b_2} \cdots x_n^{b_n}+\cwlt,$
then $f$ is effective and dominating.
\end{enumerate}
\end{lemma}

Next we recall some examples given in \cite{CPWZ1, CPWZ2, CYZ1}
that have dominating (and effective) discriminant. In the following
examples, we assume that $k$ is a field (and could be a finite field). 

\begin{example}\cite[Theorem 0.1]{CYZ1}
\label{xxex5.4}
Let 
$$A=k\langle x,y\rangle/(xy-qyx-1)$$ 
where $1\neq q$ is an $n$th root of unity. Its center is $C=k[x^n,y^n]$ 
and $A$ is free over $C$ of rank $n^2$. By \cite[Theorem 0.1]{CYZ1}
the discriminant of $d:=d_{n^2}(A/C)$ is of the form
$$d=_{k^{\times}} 
x^{n^2(n-1)}y^{n^2(n-1)}+\sum_{j< n^2(n-1)} a_{j} (xy)^j$$
for some $a_{ij}\in k$. This $d$ is dominating and effective by 
Lemma \ref{xxlem5.3}(6).
\end{example}

\begin{example}\cite[Example 5.1]{CPWZ1}
\label{xxex5.5}
Consider the algebra 
$$S(p):=k\langle x,y \rangle/
(y^2 x-p x y^2,y x^2+p  x^2 y)$$ 
where $p\in k^\times$. 
By \cite[(8.11)]{AS}, $S(p)$ 
is a noetherian Artin-Schelter regular domain of global 
dimension 3, which is of type $S_2$ in the classification 
given in \cite{AS}.  Note that $S(p)$ is 3-Koszul (so not
Koszul). Set $A=S(1)$. 
One can check that the center of $A$ is the commutative 
polynomial subring  $C:=k[x^4, y^2,\Omega]$ where
$\Omega= (xy)^2+(yx)^2$. 
As a $C$-module, $A$ is free of rank 16. A direct computation 
shows that
$$
d:= d_{16}(A/C)=_{k^\times} (x^4)^8 
(\Omega^2+4x^4 y^4)^8.
$$
In  the algebra $A$, $d$ has different presentations
$$
(x^4)^8  (\Omega^2+4x^4 y^4)^8
=(x^4)^8 (xy+i yx)^{32}=(x^4)^8 (xy-i yx)^{32}
$$
where $i^2=-1$. One can easily show that $d$ is dominating and effective
(and see Lemma \ref{xxlem5.3}(5)).
\end{example}

\begin{example} \cite[Example 5.6]{CPWZ1}
\label{xxex5.6}
Suppose $2$ is invertible in $k$.
Let $A$ be the algebra
$$k\langle x,y\rangle/(x^2 y- yx^2, xy^2-y^2 x,
 x^6-y^2).$$
It is isomorphic to
the invariant subring $k_{-1}[x_1,x_2]^{S_2}$.
Note that $A$ is a connected graded AS Gorenstein algebra with 
$\deg x=1$ and $\deg y=3$ \cite[Example 3.1]{KKZ}. By \cite[Example 5.6]{CPWZ1}, its 
discriminant $d_4(A/C(A))$
is $d:=(xy-yx)^4$. Using the relations of $A$, one has
$$d=((xy-yx)^2)^2=((xy+yx)^2-4x^2y^2)^2=(z^2-4x^8)^2=
(z-2x^4)^2(z+2x^4)^2$$
where $z=xy+yx$. It is not easy to see whether or not  $d$ is dominating
(maybe this $d$ is not dominating). Now we show that $d$ is effective.

Consider the trivial filtration on $A$ by taking $F_0 A=A$.
Pick any two elements, still denoted by $x$ and $y$ in 
a testing algebra $T$, and assume that one of them is not
in $F_0 T$. Let $z=xy+yx$. Proceed by contradiction and assume
that an expression for $f$ is nonzero and in $F_0 T$. So we have
that $f(x,y)=(z-2x^4)^2(z+2x^4)^2\neq 0$
and in $F_0 T$. Then both $z-2x^4$ and $z+2x^4$
are in $F_0 T$. So $x^4$, and whence $x$, is in $F_0 T$.
Thus $y$ must not in $F_0 T$. By using the fact $f=(xy-yx)^4$
in $F_0 T$, one obtain that $t:=xy-yx$ is in $F_0 T$. Thus
$z=2xy-t$ is not in $F_0 T$. This implies that $z\pm 2x^4$ 
are not in $F_0 T$, which implies that $f=(z-2x^4)^2(z+2x^4)^2\neq 0$
is not in $F_0 T$, a contradiction. Therefore $f$ is effective.
\end{example}

In all above examples, since the discriminant is effective, we have that
$A$ is cancellative by Theorem \ref{xxthm0.6}. 

We are now ready to prove Theorem \ref{xxthm0.7}. We give 
an expanded version of it.

\begin{theorem}
\label{xxthm5.7} 
Let $k$ be a base commutative ring containing all $p_{ij}^{\pm 1}$
and $A$ be a skew polynomial ring $k_{p_{ij}}[x_1,\cdots,x_n]$
where each $p_{ij}$ is a root of unity. Let $C$ be the center of $A$.
The following are equivalent.
\begin{enumerate}
\item[(1)]
The full automorphism group $\Aut(A)$ is affine \cite[Definition 2.5]{CPWZ1}.
\item[(2)]
Discriminant $d_w(A/C)$ is dominating where
$w={\rm{rk}}(A/C)$. 
\item[(3)]
$C$ is a subalgebra of $k[x_1^{\alpha_1},\ldots,x_n^{\alpha_n}]$
for some $\alpha_1,\ldots, \alpha_n\geq 2$.
\item[(4)]
Discriminant $d_w(A/C)$ is effective.
\item[(5)]
$x_i$ divides $d_w(A/C)$ for all $i=1,\ldots,n$.
\item[(6)]
$A$ is $\LND^H$-rigid.
\item[(7)]
$A$ is strongly $\LND^H$-rigid.
\end{enumerate}
If, further, $k$ contains ${\mathbb Q}$, then the above are also 
equivalent to
\begin{enumerate}
\item[(8)]
$A$ is $\LND$-rigid.
\item[(9)]
$A$ is strongly $\LND$-rigid.
\end{enumerate}
\end{theorem}

\begin{proof} By \cite[Theorem 3.1]{CPWZ2}, (1), (2) and (3) are
equivalent. By \cite[Theorem 2.11]{CPWZ2}, (2) and (5) are equivalent.
By the proof of \cite[Theorem 2.11]{CPWZ2}, (4) and (5) are equivalent.

(2) $\Rightarrow$ (7): This is Theorem \ref{xxthm4.7}(1).

(7) $\Rightarrow$ (6): Clear.

(6) $\Rightarrow$ (5): If (5) fails, by \cite[Theorem 2.11]{CPWZ2},
there is a homogeneous element $f$ of degree at least 2
such that $x_i f=p_{si} f x_i$ for all $s$ (such an element
corresponds to an element in $T_s$ \eqref{E6.1.1} in the next section). Then 
$$g: x_i\to \begin{cases} x_i & i\neq s\\
x_s+ tf & s=i\end{cases}, \quad {\text{and}}\quad  t\to t$$
defines an algebra automorphism of $A[t]$ over $k[t]$.
By Lemma \ref{xxlem2.2}(3), $g=G_{\partial,t}$
for some nonzero $\partial\in \LND^H(A)$, yielding
a contradiction.

Next assume that $k$ contains ${\mathbb Q}$.

(7) $\Rightarrow$ (9): This follows from the fact the map
$\LND^H(A)\to \LND(A)$ is surjective.

(9) $\Rightarrow$ (8): Obvious.

(8) $\Rightarrow$ (2): \cite[Theorem 3.1]{CPWZ2}.
\end{proof}

\section{The Makar-Limanov invariant for skew polynomial rings}
\label{xxsec6}

In this section we study $\ML^*(k_{p_{ij}}[x_1,\cdots,x_n])$.
We start with the following example.

\begin{example}
\label{xxex6.1} Suppose $n\geq 3$ is odd and $1\neq q$ is a root of
unity. Let $A=k_{q}[x_1,\ldots,x_n]$ where $k$ contains $q^{\pm 1}$ and 
${\mathbb Z}$. Then $\ML(A)=k$.
To see this, we first construct some locally nilpotent derivations.
Suppose the order of $q$ is $\ell$.
If $w$ is odd, let $\partial_w$ be the locally nilpotent derivation of
$A$ determined by
$$x_i\mapsto \begin{cases} 0 & {\rm{if}} \quad i\neq w\\
x_1^{\ell-1} x_2x_3^{\ell-1}\cdots \widehat{x_w^{\ell-1}}\cdots 
x_{n-2}^{\ell-1}x_{n-1}x_n^{\ell-1}& {\rm{if}} \quad i=w.\end{cases}$$
If $w$ is even, let $\partial_w$ be the locally nilpotent derivation of
$A$ determined by
$$x_i\mapsto \begin{cases} 0 & \qquad {\rm{if}} \quad i\neq w\\
x_1x_2^{\ell-1} x_3 \cdots \widehat{x_w^{\ell-1}}\cdots 
x_{n-2}x_{n-1}^{\ell-1}x_n& \qquad {\rm{if}} \quad i=w.\end{cases}$$
For any $f\in A\setminus k$, there is a $w$ and polynomials $f_i$ of
$x_1,\ldots, \widehat{x_w},\ldots,x_n$ such that 
$f=\sum_{i=0}^n f_i x_w^i$ with $f_n\neq 0$ and $n>0$. 
Then since $\partial_w(x_w)$ commutes with $x_w$ we have that
$\partial_w(x_w^i) =ix_w^{i-1}\partial_w(x_i)$ and so
$\partial_w (f)=\sum_{i=1}^n f_i i \partial_w(x_w) x_w^{i-1}\neq 0$.
Therefore $f\not\in \ML(A)$ and the assertion follows.
\end{example}

Fix a parameter set $\{p_{ij}\mid 1\leq i<j\leq n\}\subseteq k^*$ 
and for $i>j$ define $p_{ij}=p_{ji}^{-1}$ and define $p_{ii}=1$ 
for all $i$. For $s\in \{1,\ldots ,n\}$, recall from \cite[p.757]{CPWZ2}
that 
\begin{equation}
\label{E6.1.1}\tag{E6.1.1}
T_s:=\{(d_1,\ldots, \hat{d_s},\ldots, d_n)\in {\mathbb N}^{n-1}
\mid \prod_{j=1,j\neq s}^n p_{ij}^{d_j}=p_{is}, \; \forall \; 
i\neq s\}.
\end{equation}
By \cite[Lemma 2.9]{CPWZ2},
if $(d_1,\ldots, \hat{d_s},\ldots, d_n) \in T_s$, then the equation 
\begin{equation}
\label{E6.1.2}\tag{E6.1.2}\prod_{j=1,j\neq s}^n p_{sj}^{d_j}=1.
\end{equation}
Similar to the argument in Example \ref{xxex6.1} we have the
following result.

\begin{theorem}
\label{xxthm6.2} Let $A=k_{p_{ij}}[x_1,\cdots,x_n]$ where
all $p_{ij}$ are roots of unity.
\begin{enumerate}
\item[(1)]
Suppose that $k$ contain ${\mathbb Z}$. Then
$\ML^I(A)$ is the subalgebra of $A$ generated
by $\{x_s\mid T_s=\emptyset\}$. As a consequence,
$\ML^I(A)$ is a skew polynomial ring.
\item[(2)]
$\ML^H(A)$ is the subalgebra of $A$ generated
by $\{x_s\mid T_s=\emptyset\}$. As a consequence,
$\ML^H(A)$ is a skew polynomial ring.
\end{enumerate}
\end{theorem}

\begin{proof} (1) By replacing $k$ by its 
fraction field, we may assume that $k$ is a field containing
${\mathbb Q}$. In this case, $\ML=\ML^I$.

Let $B$ be the subalgebra generated by
$\{x_s\mid T_s=\emptyset\}$. 
First we show that $\ML(A)\subset B$. Pick $f\in A\setminus B$.
Then there is a $w$ such that $T_w\neq \emptyset$ and
polynomials $f_i$ of
$x_1,\ldots, \widehat{x_w},\ldots,x_n$ such that 
$f=\sum_{i=0}^n f_i x_w^i$ with $f_n\neq 0$ and $n>0$.
Since $T_w\neq \emptyset$, we can pick 
$(d_1,\ldots, \hat{d_w},\ldots, d_n)\in
T_w$. Using this element, we
define a locally nilpotent derivation $\partial_w$ as follows:
\begin{equation}
\label{E6.2.1}\tag{E6.2.1}
\partial_w: x_i\mapsto \begin{cases} 0 & \qquad {\rm{if}} \quad i\neq w\\
x_1^{d_1}x_2^{d_2} \cdots \widehat{x_w}\cdots 
x_{n-1}^{d_{n-1}}x_n^{d_n}& \qquad {\rm{if}} \quad i=w.\end{cases}
\end{equation}
It follows from \eqref{E6.1.2} that $x_w$ 
commutes with $\partial_w(x_w)$. Then 
$$\partial_w (f)=\sum_{i=1}^n f_i i \partial_w(x_w) x_w^{i-1}\neq 0.$$
So $f$ is not in $\ML(A)$. This implies that $\ML(A)\subset B$.

To finish the argument, one needs to show that $\partial(B)=0$ for 
every $\partial\in \LND(A)$. Or it suffices to show that $\partial(x_s)=0$
for all $s$ satisfying $T_s=\emptyset$. 

So we need to prove the following claim:  if $\partial\in \LND(A)$ and 
$T_s=\emptyset$, then $\partial(x_s)=0$. 
Let $g$ be the automorphism of the $k[t]$-algebra $A[t]$ of the form
$\exp(t \partial)$. Let $C$ be the center of $A[t]$. Since
$A[t]$ is also a skew polynomial ring over the base ring 
$k[t]$, by the proof of \cite[Theorem 2.11(1)]{CPWZ2}, 
$d_v(A[t]/C)$ is a monomial (where $v$ is taken to be 
the rank $A[t]$ over $C$). By \cite[Theorem 2.11(2)]{CPWZ2},
$d_v(A[t]/C)=\prod_{\{s\mid T_s=\emptyset\}} x_s^{a_s}$ for
some $a_s>0$. Since every automorphism preserves
$d_v(A[t]/C)$, $g(x_s)$ has $t$-degree 0. This implies
that $\partial(x_s)=0$, as required.

(2) The proof of (2) is similar to the proof of part (1).
The main difference is to replace locally nilpotent derivations by 
locally nilpotent higher derivations. We only provide a sketch here.

Let $w$ be the integer such that $T_{w}\neq \emptyset$. We claim that there 
is a locally nilpotent higher derivation $\{\Delta^n_w\}_{n=0}^{\infty}$ such that
$\Delta^1_{w}$ is the derivation $\partial_{w}$ defined in \eqref{E6.2.1}
(and $\Delta^0_{w}$ is the identity by default). For each $n\geq 2$, 
the $k$-linear map $\Delta_w^n$ is determined by 
$$\Delta^n_w: f \; x_w^{m} \longrightarrow 
\begin{cases} {m\choose n} \; f \; (\partial_w(x_w))^n x_w^{m-n} & {\rm{if}} \;\; 
m\geq n\\
0 & {\rm{otherwise,}}
\end{cases}
$$
for any $f$ being a polynomial of $x_1,\cdots, \widehat{x_w},\cdots, x_n$.
By an easy  combinatorial computation, $\{\Delta^n_w\}_{n=0}^{\infty}$ is a locally 
nilpotent higher derivation. Replacing $\partial_w$ by 
$\{\Delta^n_w\}_{n=0}^{\infty}$, the first half of the proof of part (1) can 
be recycled. The second half of the proof can be copied when $\exp(t\partial)$ 
is replaced by $G_{\partial, t}$. Details are omitted.
\end{proof}

\section{Mod-$p$ reduction}
\label{xxsec7}
In this section we introduce a method that deals
with the {\bf ZCP} for certain non-PI algebras.
We start with a temporary definition.

\begin{definition}
\label{xxdef7.1}
Let $A$ be a $k$-algebra that is free over $k$. Fix a $k$-basis 
$\{x_i\}_{i\in I}$ of $A$. Let $K$ be a subring of $k$. We say 
a subring $B\subset A$ is a $K$-order of $A$ if $\{x_i\}_{i\in I}$
is a $K$-basis of $B$.
\end{definition}

\begin{lemma}
\label{xxlem7.2} Let $K$ be a commutative domain and $A$ be a 
$K$-algebra with $K$-basis $\{x_i\}_{i\in I}$. 
Assume that $K$ is affine over ${\mathbb Z}$.
Suppose that, for every quotient field $F:=K/{\mathfrak m}$,
$A\otimes_K F$ is $\LND^H$-rigid 
{\rm{(}}respectively, strongly $\LND^H$-rigid{\rm{)}}.
Then $A$ is $\LND^H$-rigid 
{\rm{(}}respectively, strongly $\LND^H$-rigid{\rm{)}}.
\end{lemma}

\begin{proof} Let $d$ be a non-negative integer. We only
need to show that 
$$A\subseteq \ML^H(A[t_1,\ldots,t_d])$$ 
if the above holds when replacing $A$ by $A_F:=A\otimes_K F$ for 
all quotient field $F=K/{\mathfrak m}$. We proceed by contradiction.
If $\partial:=\{\partial_j\}_{j\geq 0}\in \LND^H(A[t_1,\ldots,t_d])$ 
and $\partial(f)\neq 0$ for some $f\in A$. Then there is some 
$j\geq 0$ such that 
$\partial_j(f)=\sum_{i\in I} c_i x_i$ where some $c_{i_0}$ is 
nonzero. Consider $K_1=K[c_{i_0}^{-1}]$ and take a quotient
field $F$ of $K_1$. Then $F$ is finite and $F$ is a quotient 
field of $K$ as well. Remember that $c_{i_0}$ is invertible in $K_1$,
whence invertible in $F$. Write 
$\partial_F= (\partial\otimes_K F)$. It is clear that 
$\partial_F$ is in $\LND^H(A_F[t_1,\ldots,t_d])$ since 
$G_{\partial_F, t}$ is the automorphism 
$G_{\partial,t}\otimes_K F\in \Aut(A_F[t_1,\ldots,t_d][t])$. 
Let $f'$ be the image of $f$ in $A_F$. Then 
$(\partial_F)_j(f')=\sum_{i\in I} c'_i x_i\neq 0$ where
$c'_i$ are image of $c_i$ in $F$. This contradicts the 
fact that $\ML^H(A_F[t_1,\ldots,t_d])=A_F$. Therefore the 
assertion follows.
\end{proof}

\begin{definition}
\label{xxdef7.3} Let $A$ be a $K$-algebra with a $K$-basis
$\{x_i\}_{i\in I}$. We call $\{x_i\}$ \emph{manageable} if
for each $w\geq 0$ and each $K$-algebra 
automorphism $G\in \Aut(A[y_1,\cdots,y_w])$ there is 
an affine ${\mathbb Z}$-subalgebra $K_1\subset K$
such that $B:=\oplus_{i\in I} K_1 x_i$ is a $K_1$-order 
and $G$ is induced from an automorphism of $B[y_1,\cdots,y_w]$.
\end{definition}

\begin{lemma}
\label{xxlem7.4} Let $K$ be a commutative domain and $A$ be a 
finitely generated $K$-algebra with a manageable $K$-basis 
$\{x_i\}_{i\in I}$. If, for every affine ${\mathbb Z}$-subalgebra 
$K_1\subset K$, there is an affine ${\mathbb Z}$-subalgebra 
$K_2 \subset K$ containing $K_1$ such that $B:=\oplus_{i\in I} K_2 x_i$ is 
a $K_2$-order of $A$ and that $B$ is $\LND^H$-rigid 
{\rm{(}}respectively, strongly $\LND^H$-rigid{\rm{)}}, then $A$ 
is $\LND^H$-rigid {\rm{(}}respectively, strongly $\LND^H$-rigid{\rm{)}}.
\end{lemma}

\begin{proof} Again, let $d$ be a non-negative integer.
We will show that 
$$A\subseteq \ML^H(A[t_1,\cdots,t_d]).$$ 
If not, pick $\partial\in \LND^H(A[t_1,\cdots,t_d])$ and $f\in A$
such that $\partial(f)\neq 0$. Write, for every $j\geq 0$, 
$\partial_j(f)=\sum_{i\in I} c_{ji} x_i$ where some $c_{j_0i_0}$ is 
nonzero. Consider $G=G_{\partial,t}\in 
\Aut(A[t_1,\cdots,t_d][t])$. Since $\{x_i\}_{i\in I}$
is manageable, $G=H\otimes_{K_2} K$ where $K_2$ is a 
finitely generated ${\mathbb Z}$-subalgebra of $K$ and 
$H\in \Aut(B[t_1,\cdots,t_d][t])$. By the hypothesis, we
may assume that $K_2$ contains all $c_{ji}$ and
$\ML^H(B[t_1,\cdots,t_d])=B$. Since $H(b)\equiv b \mod t$
for all $b\in B[t_1,\cdots,t_d]$, by Lemma \ref{xxlem2.2}(3),
$H=G_{\partial',t}$. Since $\ML^H(B[t_1,\cdots,t_d])=B$,
$\partial'(f)=0$, which implies that $\partial(f)=0$, a 
contradiction.
\end{proof}

Here is the main result of this section. 

\begin{theorem}
\label{xxthm7.5} Let $\Phi:=\{A\}$ be a collection of 
finitely generated algebras over various base commutative 
domains $K$. Assume that the following hold. 
\begin{enumerate}
\item[(1)]
Each $A$ is finitely generated over $K$ and has a manageable $K$-basis.
\item[(2)]
If $A\in \Phi$ where $A$ is a $K$-algebra and $K$ is an affine
${\mathbb Z}$-algebra, then $A\otimes_K F$ is in $\Phi$
where $F$ is a finite quotient field of $K$.
\item[(3)]
If $A\in \Phi$ where $A$ is a $K$-algebra, then for every 
affine ${\mathbb Z}$-subalgebra $K_1\subset K$, 
there is an affine ${\mathbb Z}$-subalgebra 
$K_2 \subset K$ containing $K_1$ such that $A=B\otimes_{K_2} K$ 
for $B\in \Phi$ that is a $K_2$-order of $A$.
\item[(4)]
If $K$ is a finite field, then $A$ is $\LND^H$-rigid 
{\rm{(}}respectively, strongly $\LND^H$-rigid{\rm{)}}.
\end{enumerate}
Then every $A$ in $\Phi$ is $\LND^H$-rigid 
{\rm{(}}respectively, strongly $\LND^H$-rigid{\rm{)}}.
\end{theorem}

\begin{proof} By Lemma \ref{xxlem7.2} and hypothesis
(4), $A$ is $\LND^H$-rigid 
{\rm{(}}respectively, strongly $\LND^H$-rigid{\rm{)}}
if $K$ is affine over ${\mathbb Z}$. Then by
Lemma \ref{xxlem7.4} and hypothesis (3), every 
$A$ in $\Phi$ is $\LND^H$-rigid 
{\rm{(}}respectively, strongly $\LND^H$-rigid{\rm{)}}.
\end{proof}

Now we are ready to prove Theorem \ref{xxthm0.8}.

\begin{proof}[Proof of Theorem \ref{xxthm0.8}]
We now construct the collection $\Phi$ as follows:
a $K$-algebra $A$ is in $\Phi$ if $A$ is a finite tensor product 
(over $K$) of different copies $K_{p}[x_1,\cdots,x_n]$ when $n$ is even 
(for different values of $p\neq 1$), copies of 
$K\langle x,y\rangle/(x^2y-yx^2,y^2x+xy^2)$, and different copies 
of $K\langle x,y\rangle/(yx-qxy-1)$ 
(for different values of $q\neq 1$). We require that the base 
commutative rings $K$ are domains containing the following elements
\begin{equation}
\label{E7.5.1}\tag{E7.5.1}
2^{-1}, p^{\pm 1}, q^{\pm 1}, (p-1)^{-1}, (q-1)^{-1}
\end{equation}
for different $p$ and $q$ occurring in $A$. In particular,
\begin{equation}
\label{E7.5.2}\tag{E7.5.2}
p\neq 0, 1 \quad {\rm{and}} \quad  q\neq 0, 1 
\quad {\rm{and}} \quad  2\neq 0. 
\end{equation}
The conditions in \eqref{E7.5.2} will survive 
when passing $K$ to a finite quotient field $K/I$
due to \eqref{E7.5.1}. If $K$ is a finite
field, Example \ref{xxex4.8}(4) says that $A$ has 
dominating discriminant. By Theorems \ref{xxthm4.7}(1)
and \ref{xxthm5.2}(2), $A$ is strongly $\LND^H$-rigid.
This verified hypothesis (4) of Theorem \ref{xxthm7.5}.
Hypotheses (1,2,3) of Theorem \ref{xxthm7.5} are easy to 
verify, see Lemma \ref{xxlem7.6} below. Therefore 
every member in $\Phi$ is strongly $\LND^H$-rigid. 

Since every $A$ is a domain of finite GK-dimension,
by Theorem \ref{xxthm3.3}, $A$ is strongly cancellative. 
\end{proof}

\begin{lemma}
\label{xxlem7.6} Retain notation as in the proof of 
Theorem {\rm{\ref{xxthm0.8}}} above. 
\begin{enumerate}
\item[(1)]
Each $A$ is finitely generated over $K$ and has a manageable $K$-basis.
\item[(2)]
If $A\in \Phi$ where $A$ is a $K$-algebra and $K$ is an affine
${\mathbb Z}$-algebra, then $A\otimes_K F$ is in $\Phi$
where $F$ is a finite quotient field of $K$.
\item[(3)]
If $A\in \Phi$ where $A$ is a $K$-algebra, then for every 
affine ${\mathbb Z}$-subalgebra $K_1\subset K$, 
there is an affine ${\mathbb Z}$-subalgebra 
$K_2 \subset K$ containing $K_1$ such that $A=B\otimes_{K_2} K$ 
for $B\in \Phi$ that is a $K_2$-order of $A$.
\end{enumerate}
\end{lemma}

\begin{proof} Part (2) is clear by the definition of $\Phi$. 
We only prove parts (1) and (3).

For each $K$-algebra $A$ in $\Phi$, it is well-known that 
there is a $K$-linear basis ${\mathcal A}$ 
of $A$ that consists of a family of 
noncommutative monomials. One property of this basis is that 
it is independent of the parameters $p,q$ and independent of 
the base ring $K$. The multiplication constant with 
respect to this basis are in the ${\mathbb Z}$-subalgebra of
$K$ generated by elements in \eqref{E7.5.1}. Then (3) follows
by taking $K_2$ to be the subalgebra generated by $K_1$ and 
elements in \eqref{E7.5.1}.

For part (1), we show that the basis ${\mathcal A}$ 
used in the last paragraph is manageable. Let 
$E:=A[y_1,\cdots,y_w]$ where $A$ is in $\Phi$.
It is clear that $E$ has a $K$-basis, denoted by 
${\mathcal E}$, of the form 
$$\{ b_n y_1^{d_1}\cdots y_w^{d_w}\mid b_n\in {\mathcal A}, 
d_s\geq 0\}.$$
Since $A$ is finitely generated over $K$, so is $E$. Let
$E$ be generated by $\{f_1,\cdots,f_z\}\subseteq 
{\mathcal E}$. Let $G\in \Aut(E)$. 
Then there is an affine ${\mathbb Z}$-subalgebra $K_1$ 
of $K$ such that $G(f_s)$ and $G^{-1}(f_s)$, for $s=1,\cdots,z$, 
are all in the $K_1$-span $K_1 {\mathcal E}$. Without loss of generality 
we assume that $K_1$ contains elements in \eqref{E7.5.1}. Then 
$K_1 {\mathcal A}$ is an algebra, which is denoted by $B$. 
(By part (3), we might further assume that $B$ is in $\Phi$.) Since 
$A$ and $B$ has the ``same'' basis (over different commutative 
rings), $B$ is a $K_1$-order of $A$. By the choice of $K_1$,
$G$ restricts to an algebra automorphism $G'$ of $B[y_1,\cdots,y_w]$.
Since $E$ and $B[y_1,\cdots,y_w]$ has the ``same'' basis, $G$ is induced
form $G'$. Therefore part (1) holds.
\end{proof}

We summarize the key steps of solving {\bf ZCP} for 
noncommutative algebras similar to those in Theorem \ref{xxthm0.8}
as follows. For an algebra $A$ over a base commutative 
ring $k$ satisfying certain finiteness conditions, one uses 
reduction modulo $p$ to reduce the problem in the special case 
when $k$ is a finite field. When $k$ is a finite field, 
$A$ ends up being PI (which is true for a large class 
of quantized algebras). Then one can compute the discriminant
of $A$ over its center, say $d:=d(A/C)$. If one can verify that 
$d$ is effective (or dominating), then $A$ is strongly 
$\LND^H$-rigid by Theorems \ref{xxthm4.7}(1) and 
\ref{xxthm5.2}(2). Finally, by Theorem \ref{xxthm3.3}(1),
$A$ is strongly cancellative. So we have the following diagram.
$$\begin{CD}
{\text{$A$ with finiteness conditions}} \\
@VV{\text{reduction mod $p$}}V
\\
{\text{$A$ over a finite field $k$}}\\
@VV
{\text{computing discriminant $d$}}V \\
{\text{Checking if $d$ is effective}}\\
@VV{\text{Theorem \ref{xxthm5.2}(2)}}V\\
{\text{$A$ is strongly $\LND^H$-rigid}}\\
@VV{\text{Theorem \ref{xxthm3.3}(1)}}V\\
{\text{$A$ is strongly cancellative.}}\\
\end{CD}$$
For algebras of GK-dimension two we should 
apply Theorem \ref{xxthm0.5} directly.

\providecommand{\bysame}{\leavevmode\hbox to3em{\hrulefill}\thinspace}
\providecommand{\MR}{\relax\ifhmode\unskip\space\fi MR }
\providecommand{\MRhref}[2]{%
\href{http://www.ams.org/mathscinet-getitem?mr=#1}{#2} }
\providecommand{\href}[2]{#2}

\end{document}